\documentclass{article}

\usepackage[utf8]{inputenc}
\usepackage{amssymb}
\usepackage{makeidx}
\usepackage[english]{babel}
\usepackage{graphicx}
\usepackage{amsfonts,amsmath,amssymb,amsthm}
\usepackage{oldgerm}
\usepackage{mathrsfs}
\usepackage[active]{srcltx}
\usepackage{verbatim}
\usepackage{enumitem} 
\usepackage[toc,page]{appendix}
\usepackage{aliascnt}
\usepackage{array}
\usepackage{hyperref}
\usepackage[textwidth=4cm,textsize=footnotesize]{todonotes}
\usepackage{xargs}
\usepackage{cellspace}
\usepackage[Symbolsmallscale]{upgreek}
\usepackage{geometry}
\usepackage{array}
\usepackage{fancyhdr}
\usepackage{accents}
\usepackage{dsfont}
\usepackage{aliascnt}
\usepackage{cleveref}
\makeatletter
\newtheorem{theorem}{Theorem}
\crefname{theorem}{theorem}{Theorems}
\Crefname{Theorem}{Theorem}{Theorems}

\newtheorem{lemma}{Lemma}
\crefname{lemma}{lemma}{lemmas}
\Crefname{Lemma}{Lemma}{Lemmas}

\crefname{corollary}{corollary}{corollaries}
\Crefname{Corollary}{Corollary}{Corollaries}

\newtheorem{proposition}{Proposition}
\crefname{proposition}{proposition}{propositions}
\Crefname{Proposition}{Proposition}{Propositions}

\crefname{definition}{definition}{definitions}
\Crefname{Definition}{Definition}{Definitions}

\crefname{remark}{remark}{remarks}
\Crefname{Remark}{Remark}{Remarks}

\newtheorem{example}{Example}
\crefname{example}{example}{examples}
\Crefname{Example}{Example}{Examples}

\crefname{figure}{figure}{figures}
\Crefname{Figure}{Figure}{Figures}

\newtheorem{assumptionH}{\textbf{H}\hspace{-3pt}}
\Crefname{assumptionH}{\textbf{H}\hspace{-3pt}}{\textbf{H}\hspace{-3pt}}
\crefname{assumptionH}{\textbf{H}}{\textbf{H}}

\newtheorem{assumptionG}{\textbf{G}\hspace{-3pt}}
\Crefname{assumptionG}{\textbf{G}\hspace{-3pt}}{\textbf{H}\hspace{-3pt}}
\crefname{assumptionG}{\textbf{G}}{\textbf{G}}

\Crefname{assumptionA}{\textbf{A}\hspace{-3pt}}{\textbf{A}\hspace{-3pt}}
\crefname{assumptionA}{\textbf{A}}{\textbf{A}}

\def\Rset{\mathbb{R}}
\def\Nset{\mathbb{N}}
\def\nset{\mathbb{N}}
\def\rset{\mathbb{R}}
\def\varb{b}

\newcommand{\1}{\mathds{1}}
\newcommand{\floor}[1]{\left\lfloor #1 \right\rfloor}
\newcommand{\ceil}[1]{\left\lceil #1 \right\rceil}

\def\ie{\text{i.e.}}
\newcommand{\coint}[1]{\left[#1\right)}
\newcommand{\ocint}[1]{\left(#1\right]}
\newcommand{\ooint}[1]{\left(#1\right)}
\newcommand{\ccint}[1]{\left[#1\right]}
\newcommand{\Wienerspace}{\mathbf{W}}

\def\PP{\mathbb{P}}
\def\PE{\mathbb{E}}
\newcommandx{\CPE}[3][1=]{{\mathbb E}_{#1}\left[\left. #2 \, \right| #3 \right]} \newcommandx{\PEt}[2][1=]{\PE_{#1}\left[#2\right]}
\def\rmd{\mathrm{d}}
\def\rme{\mathrm{e}}
\def\target{\pi}
\def\dim{d}

\def\Vdot{\dot{V}}
\def\I{\mathcal{I}}

\def\Iclosed{\overline{\mathcal{I}}}
\def\unitarget{\pi}
\def\rmL{\mathrm{L}}
\def\Lp{\rmL^p}

\def\G{\mathcal{G}}

\newcommand{\norm}[1]{\left\Vert #1 \right \Vert}
\newcommand{\plusinfty}{+ \infty}
\newcommand{\defEns}[1]{\left \{ #1 \right \}}
\newcommand{\abs}[1]{\left\vert #1 \right\vert}

\def\eqsp{\,}

\def\wrt{with respect to}
\newcommand{\parenthese}[1]{\left( #1 \right)}
\newcounter{hypH}

\newcommand{\eqdef}{\ensuremath{:=}}
\newcommand{\setAccept}{\mathcal{A}}
\def\iid{\operatorname{i.i.d.}}
\newcommand{\sachant}[1]{\left| #1 \right.}
\newcommand{\setDisconDotV}{\mathcal{D}_{\dot{V}}}

\newcommandx\sequence[3][2=,3=]
{\ifthenelse{\equal{#3}{}}{\ensuremath{\{ #1_{#2}\}}}{\ensuremath{( #1_{#2}, \eqsp #2 \in #3 )}}}
\newcommandx{\sequencen}[2][2=n\in\nset]{\ensuremath{( #1, \eqsp #2 )}}
\newcommandx\dsequence[4][3=k,4=\zset]{\ensuremath{( (#1_{#3}, #2_{#3}), \eqsp #3 \in #4 )}}
\newcommandx{\as}[1][1=\PP]{\ensuremath{#1\--\mathrm{a.s.}}}
\newcommand{\eqspp}{\ \ }

\def\canonicalFiltration{\mathscr{B}}

\def\barX{\bar{X}}
\def\barY{\bar{Y}}
\def\barS{\bar{S}}

\newcommand\proba[1]{\PP\left[ #1 \right]}
\newcommand\expe[1]{\PE\left[ #1 \right]}
\def\osc{\operatorname{osc}}

\def\mcf{\mathcal{F}}

\def\sign{\mathrm{sign}}

\def\pdf{\pi}
\def\sqpdf{\xi}

\newcommandx{\lp}[1][1=p]{\ensuremath{\mathrm{L}^{{#1}}}}

\def\coeffLasso{\lambda}
\def\smoothPotLasso{U}
\def\Erm{\mathrm{E}}
\def\Dcal{\mathcal{D}}
\def\Jrm{\mathrm{J}}
\def\Lrm{\mathrm{L}}
\def\Xsf{\mathsf{X}}
\def\Ysf{\mathsf{Y}}
\def\Usf{\mathsf{U}}
\def\constSet{\mathrm{r}}
\def\Vgamma{V_{\upgamma}}
\def\Vgammadot{\dot{V}_{\upgamma}}
\def\pigamma{\pi_{\upgamma}}
\def\Vbeta{V_{\upbeta}}
\def\Vbetadot{\dot{V}_{\upbeta}}
\def\pibeta{\pi_{\upbeta}}
\def\ESJD{\operatorname{ESJD}}
\def\rma{\mathrm{a}}
\def\generator{\Lrm}

\begin{document}

\author{Alain Durmus\footnotemark[1]\and Sylvain Le {C}orff\footnotemark[2]\and Eric Moulines\footnotemark[3]\and Gareth O. Roberts\footnotemark[4]}
 
\footnotetext[1]{LTCI, CNRS and T\'el\'ecom ParisTech.}
\footnotetext[2]{Laboratoire de Math\'ematiques d'Orsay, Univ. Paris-Sud, CNRS, Universit\'e Paris-Saclay.}
\footnotetext[3]{Centre de Math\'ematiques Appliqu\'ees, Ecole Polytechnique.}
\footnotetext[4]{University of Warwick, Department of Statistics.}

\title{Optimal scaling of the Random Walk Metropolis algorithm under $\Lp$  mean differentiability}
\date{}
\lhead{Durmus et al.}
\rhead{Optimal scaling under $\Lp$  mean differentiability}

\maketitle

\begin{abstract}
This paper considers the optimal scaling problem for high-dimensional random walk Metropolis algorithms for densities which are differentiable in  $\Lp$ mean but which may be irregular at some points (like the Laplace density for example) and / or are supported on an interval. Our main result is the weak convergence of the Markov chain (appropriately rescaled in time and space) to a Langevin diffusion process  as the dimension $\dim$ goes to infinity. Because the log-density might be non-differentiable, the limiting diffusion could be singular. The scaling limit is established under assumptions which are much weaker than the one used in the original derivation of \cite{roberts:gelman:gilks:1997}. This result has important practical implications for the use of random walk Metropolis algorithms in Bayesian frameworks based on sparsity inducing priors.
\end{abstract}

\section{Introduction}
\label{sec:intro}
A wealth of contributions have been devoted to the study of the behaviour  of high-dimensional Markov chains.
One of the most powerful approaches for that purpose is the scaling analysis, introduced by \cite{roberts:gelman:gilks:1997}.
Assume that the target distribution has a density \wrt~the $\dim$-dimensional
Lebesgue measure given by:
\begin{equation}
 \label{eq:targetiid}
\target^{\dim}(x^{\dim}) = \prod_{i=1}^{\dim} \unitarget(x_i^{\dim}) \eqsp.
\end{equation}
The Random Walk Metropolis-Hastings (RWM) updating scheme was first applied in \cite{metropolis:rosenbluth:teller:1953} and proceeds as follows.  Given the current state $X^{\dim}_k$, a new value  $Y_{k+1}^{\dim} = (Y_{k+1,i}^{\dim})_{i=1}^{\dim}$ is obtained by moving independently each coordinate, \ie\  $Y^\dim_{k+1,i}= X^\dim_{k,i} + \ell \dim^{-1/2} Z_{k+1}^\dim$   where $\ell > 0$ is a scaling factor and $\left(Z_{k}\right)_{k\ge 1}$ is a sequence of independent and identically distributed ($\iid$)  Gaussian random variables. Here $\ell$ governs the overall size of the proposed jump and plays a crucial role in determining the efficiency of the algorithm. The proposal is then accepted or rejected according to the acceptance probability $\alpha(X_k^\dim,Y_{k+1}^\dim)$ where $\alpha(x^\dim,y^{\dim})= 1 \wedge \target^{\dim}(y^\dim)/\target^\dim(x^\dim)$. If the proposed value is accepted it becomes the next current value, otherwise the current value is left unchanged:
\begin{align}
\label{eq:proposaliid}
X_{k+1}^{\dim} &= X_k ^{\dim} + \ell \dim^{-1/2} Z^{\dim}_{k+1} \1_{\setAccept^{\dim}_{k+1}} \eqsp, \\
\label{eq:setAccept}
\setAccept_{k+1}^{\dim} &= \left\lbrace U_{k+1} \leq \prod_{i=1}^{d} \unitarget(X_{k,i}^{\dim} +
\ell \dim^{-1/2} Z^{\dim}_{k+1,i}) / \unitarget( X_{k,i}^{\dim}) \right\rbrace \eqsp,
\end{align}
where  $\left(U_{k}\right)_{k\ge 1}$ of $\iid$ uniform
random variables on $\ccint{0,1}$ independent of $\left(Z_k\right)_{k \geq 1}$.

Under certain regularity assumptions on $\pi$, it has been proved in \cite{roberts:gelman:gilks:1997} that if the  $X_0^\dim$ is distributed under  the stationary distribution $\target^\dim$, then each component of $(X^{\dim}_k)_{k\ge 0}$ appropriately rescaled in time converges weakly to a Langevin diffusion process
with invariant distribution $\pi $ as $\dim \to +\infty$.

This result allows to compute the asymptotic mean acceptance rate and
to derive a practical rule to tune the factor $\ell$. It is shown in \cite{roberts:gelman:gilks:1997} that the speed of the limiting diffusion has a function of $\ell$ has a unique maximum. The corresponding mean acceptance rate in stationarity
is equal to $0.234$.

These results have been derived for target distributions of the form (\ref{eq:targetiid}) where $\unitarget(x) \propto \exp(-
V(x))$ where $V$ is three-times continuously differentiable. Therefore, they do not cover the cases where the
target density is continuous but not smooth, for example the Laplace distribution which plays a key role as a sparsity-inducing prior in high-dimensional Bayesian inference.

The aim of this paper is to extend the scaling results for the RWM algorithm introduced in the seminal paper \cite[\Cref{theo:diffusion_limit_RMW}]{roberts:gelman:gilks:1997} to
densities which are absolutely continuous densities differentiable in $\Lp$ mean
(DLM) for some $p \geq 2$ but can be either non-differentiable at some points or are supported on an interval. As shown
in \cite[Section~17.3]{lecam:1986}, differentiability of the square
root of the density in $\mathrm{L}^2$ norm implies a quadratic
approximation property for the log-likelihood known as local
asymptotic normality. As shown below, the DLM permits 
the quadratic expansion of the log-likelihood without
paying the twice-differentiability price usually demanded by such a
Taylor expansion (such expansion of the log-likelihood plays a key
role in \cite{roberts:gelman:gilks:1997}).

The paper is organised as follows. In \Cref{sec:R} the target density $\pi$ is assumed to be
positive on $\rset$. \Cref{theo:result_acceptance_rate_RWM} proves that under the
DLM assumption of this paper, the average acceptance rate and the expected square jump distance 
are the same as in \cite{roberts:gelman:gilks:1997}. \Cref{theo:diffusion_limit_RMW} shows that under
the same assumptions the rescaled in time Markov chain produced by the
RWM algorithm converges weakly to a Langevin diffusion. We show that these results may be applied to a density of the form $\pi(x) \propto \exp(-\lambda \abs{x} + \smoothPotLasso(x))$, where $\lambda\ge 0$ and $\smoothPotLasso $ is a smooth function. In \Cref{sec:I}, we focus on the case where $\pi$ is supported only on an open interval of $\rset$. Under appropriate assumptions, \Cref{theo:result_acceptance_rate_RWM:G} and \Cref{theo:diffusion_limit_RMW:G} show that the same asymptotic results (limiting average acceptance rate and limiting Langevin diffusion associated with $\pi$) hold.  We apply our results to Gamma and Beta distributions. The proofs are postponed to \Cref{sec:proofs} and \Cref{sec:proofs:G}.


\section{Positive Target density on $\Rset$}
\label{sec:R}
\label{subsec:result:R}
The key of the proof of our main result will be to show  that  the acceptance ratio and the expected square jump distance
converge to a finite and non trivial limit. In the original proof of \cite{roberts:gelman:gilks:1997}, the density of the product form
(\ref{eq:targetiid}) with
\begin{equation}
\label{eq:definition-unitarget}
\unitarget(x) \propto \exp(-V(x))
\end{equation}
is three-times continuously differentiable and  the acceptance ratio is expanded using the usual pointwise Taylor formula. More precisely,  the log-ratio of the density evaluated at the proposed value and at the current state is given by $\sum_{i=1}^{\dim} \Delta V^{\dim}_{i}$ where
\begin{equation}
\label{eq:Delta_V}
\Delta V^{\dim}_{i} = V\left(X_{i}^{\dim}\right) - V\left(X_{i}^{\dim} + \ell\dim^{-1/2} Z^{\dim}_{i}\right)\eqsp,
\end{equation}
and where $X^{\dim}$ is  distributed according to $\target^{\dim}$ and $Z^{d}$ is a $d$-dimensional standard Gaussian random variable independent of $X$.
Heuristically, the two leading terms are $\ell \dim^{-1/2} \sum_{i=1}^{\dim}
\dot{V}(X^{\dim}_i) Z^{\dim}_i$ and $\ell^2 d^{-1} \sum_{i=1}^{\dim} \ddot{V}(X^{\dim}_i)
(Z_i^{\dim})^2/2$, where $\dot{V}$ and $\ddot{V}$ are the first and second derivatives of $V$, respectively.  By the
central limit theorem, this expression converges in distribution to a zero-mean Gaussian random variable with variance $\ell^2I$ where
\begin{equation}
\label{eq:translation-score}
I = \int_{\rset}  \Vdot^2(x)\pi(x) \rmd x \eqsp.
\end{equation}
Note that  $I$ is the Fisher information associated with the translation model
$\theta \mapsto \pdf(x+\theta)$ evaluated at $\theta=0$. Under appropriate technical conditions, the second term converges almost surely to $-\ell^2I/2$. Assuming that these limits exist, the acceptance ratio in the RWM algorithm  converges to $\PE[ 1 \wedge \exp(\mathsf{Z})]$ where $\mathsf{Z}$ is a Gaussian random variable with mean $-\ell^2 I/2$ and variance $\ell^2 I$; elementary computations show that $\PE[ 1 \wedge \exp(\mathsf{Z})]= 2\Phi(-\ell/2 \, \sqrt{I})$,
where $\Phi$ stands for the cumulative distribution function of a standard normal distribution.

For $t \geq 0$, denote by $Y_{t}^{\dim}$ the linear interpolation of the Markov chain  $(X_k^{\dim})_{k\ge 0}$ after time rescaling:
\begin{align}
\label{eq:defY}
Y_{t}^{\dim} &= \left(\lceil \dim \, t\rceil - \dim \, t\right)X_{\floor{\dim \, t}}^{\dim} + \left(\dim \, t - \lfloor\dim \, t\rfloor\right)X_{\ceil{\dim \, t}}^{\dim} \\
&= X_{\floor{\dim \, t}}^{\dim} + \left(\dim \, t - \lfloor\dim \, t\rfloor\right) \ell
\dim^{-1/2} Z^{\dim}_{\lceil \dim \, t\rceil} \1_{\setAccept^{\dim}_{\lceil \dim \, t\rceil}} \eqsp,
\end{align}
where $\lfloor \cdot \rfloor$ and $\lceil \cdot \rceil$ denote the lower and the upper integer part functions.
Note that for all $k\ge 0$, $Y_{k/\dim}^{\dim} = X_{k}^{\dim}$. Denote by $(B_t, t \geq 0)$ the standard Brownian motion.
\begin{theorem}[\cite{roberts:gelman:gilks:1997}]
\label{theo:roberts:gelman:gilks}
Suppose that the target $\pi^\dim$ and the proposal distribution are given by \eqref{eq:targetiid}-\eqref{eq:definition-unitarget} and \eqref{eq:proposaliid} respectively. Assume that
\begin{enumerate}[label=(\roman*)]
\item $V$ is twice continuously differentiable and $\dot{V}$ is Lipshitz continuous ;
\item $\PE[(\dot{V}(X))^8] < \infty$ and $\PE[(\ddot{V}(X))^4] < \infty$ where $X$ is
  distributed according to $\pi$.
\end{enumerate}
Then
$(Y_{t,1}^\dim, t \geq 0)$, where $Y^{\dim}_{t,1}$ is the first component of the vector $Y^{\dim}_t$ defined in \eqref{eq:defY}, converges weakly in the Wiener space (equipped with the uniform topology) to the Langevin diffusion
\begin{equation}
\label{eq:langevin_limit}
\rmd Y_t  = \sqrt{h(\ell)} \rmd B_t - \frac{1}{2} h(\ell) \Vdot(Y_t) \rmd t \eqsp,
\end{equation}
where $Y_0$ is distributed according to $\pdf$, $h(\ell)$ is given by
\begin{equation}
\label{eq:defhK}
h (\ell) = 2 \ell^2 \Phi\left(-\frac{\ell}{2} \, \sqrt{I} \right)  \eqsp,
\end{equation}
and $I$ is defined in \eqref{eq:translation-score}.
\end{theorem}
Whereas $V$ is assumed to be twice continuously differentiable, the dual representation of the Fisher information
$-\PE[\ddot{V}(X)]= \PE[ (\dot{V}(X))^2]=I$ allows us to remove in the statement of the
theorem all mention to the second derivative of $V$, which hints that two derivatives might not
really be required. For all $\theta,x  \in \rset$, define
\begin{equation}
  \label{eq:def_xitheta}
\sqpdf_\theta(x) =\sqrt{\pdf(x+\theta)} \eqsp,
\end{equation}
For $p \geq 1$, denote $\norm{f}_{\pdf,p}^p= \int |f(x)|^p \pdf(x) \rmd x$.
Consider the following assumptions:
\begin{assumptionH}
\label{assum:diff:quadratic}
There exists a measurable function $\Vdot:\Rset\to\Rset$ such that:
\begin{enumerate}[label=(\roman*)]
\item
\label{hyp:mean_square_deriv2}
There exist $p >4$, $C>0$ and $\beta>1$ such that for all $\theta\in\rset$,
\[
\norm{V(\cdot+\theta)-V(\cdot)-\theta\Vdot(\cdot)}_{\pdf,p}  \le C|\theta|^{\beta}\eqsp.
\]
\item
\label{assum:X6}
The function $\Vdot$ satisfies $\norm{\Vdot}_{\pdf,6}<+\infty$.
\end{enumerate}
\end{assumptionH}

\begin{lemma}
\label{lem:DQM}
Assume \Cref{assum:diff:quadratic}. Then, the family of densities $\theta \to \pi(\cdot +\theta)$ is Differentiable in Quadratic Mean (DQM) at $\theta=0$ with derivative $\Vdot$, \ie~there exists $C>0$ such that for all $\theta\in\rset$,
\[
\parenthese{\int_{\Rset} \left(\sqpdf_\theta(x)-\sqpdf_0(x) + \theta \Vdot(x)\sqpdf_0(x)/2\right)^2 \rmd x}^{1/2} \le C|\theta|^{\beta}\eqsp,
\]
where $\xi_\theta$ is given by \eqref{eq:def_xitheta}.
\end{lemma}
\begin{proof}
The proof is postponed to \Cref{sec:proof-DQM}.
\end{proof}
The first step in the proof is to show that the acceptance ratio $\PP \left( \setAccept_1^d \right) = \PE(  1 \wedge \exp\{\sum_{i=1}^{\dim} \Delta V^{\dim}_{i}\} )$,
and the expected square jump distance
$\PE[(Z_{1}^d)^2\{  1 \wedge \exp(\sum_{i=1}^{\dim} \Delta V^{\dim}_{i}) \}]$
both converge to a finite  value.  To that purpose, we consider
\begin{equation*}
\Erm^{\dim}(q) = \PEt[]{\left(Z_{1}^d\right)^q\left|  1 \wedge \exp\left(\sum_{i=1}^{\dim} \Delta V^{\dim}_{i}\right) - 1 \wedge \exp(\upsilon^\dim) \right|} \eqsp,
\end{equation*}
where $\Delta V^{\dim}_{i}$ is given by \eqref{eq:Delta_V},
\begin{align}
\label{eq:definition-W-d}
\upsilon^\dim &=  -\ell \dim^{-1/2}Z_{1}^{\dim}\Vdot(X_{1}^{\dim}) +\sum_{i=2}^{\dim} b^{\dim}(X_i^{\dim},Z_i^{\dim}) \\
\label{eq:bdi}
b^{\dim}(x,z) &= -\frac{\ell
  z}{\sqrt{\dim}}\Vdot(x)+\PE\left[2\zeta^{\dim}(X_1^{\dim},Z_1^{\dim})\right]
-  \frac{\ell^2 }{4\dim}\Vdot^2(x) \eqsp, \\
\label{eq:def:zeta}
\zeta^{\dim}(x,z)&= \exp \left\{\left(V \left(x \right)-V\left(x+\ell \dim^{-1/2} z
    \right)\right)/2\right\}-1  \eqsp.
\end{align}

\begin{proposition}
\label{lem:approx_ratio}
Assume \Cref{assum:diff:quadratic} holds. Let $X^{\dim}$ be a random variable distributed
according to $\target^{\dim}$ and $Z^{\dim}$ be a zero-mean standard Gaussian random variable, independent of $X^{\dim}$. Then, for any $q \geq 0$,
 $\lim_{\dim\to\plusinfty} \Erm^{\dim}(q) = 0$.
\end{proposition}
\begin{proof}
The proof is postponed to \Cref{proof:lem:approx_ratio}.
\end{proof}
\Cref{lem:approx_ratio} shows that it is enough to consider $\upsilon^d$ to
analyse the asymptotic behaviour of the acceptance ratio and the
expected square jump distance as $d \to \plusinfty$. By the central
limit theorem, the term $ -\ell \sum_{i=2}^d
(Z_{i}^{\dim}/\sqrt{\dim})\Vdot(X_{i}^{\dim})$ in \eqref{eq:definition-W-d} converges in
distribution to a zero-mean Gaussian random variable with variance
$\ell^2 I$, where $I$ is defined in \eqref{eq:translation-score}. By
\Cref{lem:mean:zeta} (\Cref{proof:theo:result_acceptance_rate_RWM}),
the second term, which is
$\dim\,\PE\left[2\zeta^{\dim}(X_{1}^{\dim},Z_{1}^{\dim})\right] =
-\dim\,\PE\left[(\zeta^{\dim}(X^{\dim}_{1},Z_{1}^{\dim}))^2\right] $
converges to $-\ell^2I/4$. The last term converges in probability to
$-\ell^2I/4$. Therefore, the two last terms plays a similar role in
the expansion of the acceptance ratio as the second derivative of $V$
in the regular case.
\begin{theorem}
\label{theo:result_acceptance_rate_RWM}
Assume \Cref{assum:diff:quadratic} holds. Then,
$\lim_{\dim \to \plusinfty } \PP\left[\setAccept_1^\dim \right] = a(\ell)$, where $a(\ell) = 2 \Phi(-\sqrt{I} \ell /2)$.
\end{theorem}
\begin{proof}
The proof is postponed to \Cref{proof:theo:result_acceptance_rate_RWM}.
\end{proof}
The second result of this paper is that the sequence $\{(Y_{t,1}^{\dim})_{t\ge 0}, \dim\in \Nset^{\star}\}$  defined by \eqref{eq:defY} converges weakly to a Langevin diffusion. Let $(\mu_d)_{d\ge 1}$ be the sequence of distributions of $\{(Y_{t,1}^{\dim})_{t\ge 0}, \dim\in \Nset^{\star}\}$.
 \begin{proposition}
\label{prop:tight}
Assume \Cref{assum:diff:quadratic} holds. Then, the sequence $\left(\mu_d\right)_{d\ge 1}$ is tight in $\Wienerspace$.
\end{proposition}
\begin{proof}
The proof is adapted from \cite{jourdain:lelievre:miasojedow:2015}; it is postponed to \Cref{sec:proof:tightness}.
\end{proof}
By the Prohorov theorem, the tightness of  $\left(\mu_{\dim}\right)_{\dim\ge 1}$ implies that this sequence has a  weak limit point. We now prove that any limit point is the law of a solution to \eqref{eq:langevin_limit}. For that purpose, we use the equivalence between the weak formulation of stochastic differential equations and martingale problems. The generator $\Lrm$ of the Langevin diffusion  \eqref{eq:langevin_limit} is given, for all $\phi \in C^2_c(\rset,\rset)$, by
\begin{equation}
\label{eq:def-generator}
\Lrm \phi(x) = \frac{h(\ell)}{2}\left(-\Vdot(x)\dot{\phi}(x) + \ddot{\phi}(x)\right) \eqsp,
\end{equation}
where for $k \in \nset$ and $I$ an open subset of $\rset$, $C^k_c(I, \rset)$ is the space of $k$-times differentiable functions with compact support, endowed with the topology of uniform convergence of all derivatives up to order $k$. We set $C^{\infty}_c(I, \rset)= \bigcap_{k=0}^\infty C^k_c(I, \rset)$ and $\Wienerspace= C(\rset_+,\rset)$. The canonical process is denoted by $(W_t)_{t\geq 0}$ and $(\canonicalFiltration_t)_{t \geq 0}$ is the associated filtration. For any probability measure $\mu$ on $\Wienerspace$, the expectation with respect to $\mu$ is denoted by $\PE^{\mu}$. A probability measure
$\mu$ on $\Wienerspace$ is said to solve the martingale problem associated with \eqref{eq:langevin_limit} if
the pushforward of $\mu$ by $W_0$ is  $\pdf$ and if for all $\phi \in C^{\infty}_c(\rset,\rset)$,
the process
$$\left(\phi(W_t) - \phi(W_{0}) - \int_{0}^t \Lrm \phi(W_u) \rmd u\right)_{ t \geq 0} $$
 is a martingale with respect to $\mu$ and the filtration $(\canonicalFiltration_t)_{t
   \geq 0}$, \ie~if for all $s,t \in \rset_+, s \leq t$, \eqspp \as[\mu]
$$\PE^{\mu}\left[\phi(W_t) -\phi(W_0)- \int_{0}^t \Lrm \phi(W_u) \rmd u \middle| \canonicalFiltration_s\right] = \phi(W_s)-\phi(W_0) - \int_{0}^s \Lrm \phi(W_u) \rmd s \eqsp.
$$
\begin{assumptionH}
\label{assum:Vdot}
The function $\Vdot$ is continuous on $\rset$ except on a Lebesgue-negligible set $\setDisconDotV$ and is bounded on all compact sets of $\rset$.
\end{assumptionH}
If $\Vdot$ satisfies \Cref{assum:Vdot}, \cite[Lemma 1.9, Theorem 20.1 Chapter 5]{rogers:williams:2000-2} show that  any solution to  the martingale problem associated with \eqref{eq:langevin_limit} coincides with the law of a solution to the SDE \eqref{eq:langevin_limit}, and conversely. Therefore, uniqueness in law of weak solutions to \eqref{eq:langevin_limit} implies uniqueness of the solution of the martingale problem.
\begin{proposition}
\label{propo:reduction_martingale_problem}
Assume \Cref{assum:Vdot} holds. Assume also that for all $\phi \in C_c^{\infty}(\rset,\rset)$, $m \in \nset^*$, $g : \rset^m \to \rset$ bounded and continuous, and $0 \leq t_1 \leq \dots \leq t_m \leq s \leq t$:
 \begin{equation}
 \label{eq:reduction_martingale_problem}
 \lim_{d \to \plusinfty} \PE^{\mu_d}\left[ \left(\phi\left(W_t\right) - \phi\left( W_s \right) - \int_s^t \Lrm \phi \left(W_u \right) \rmd u  \right) g\left(W_{t_1},\dots,W_{t_m} \right) \right] =0 \eqsp.
\end{equation}
Then, every limit point of the sequence of probability measures $(\mu_{\dim})_{\dim\ge 1}$ on $\Wienerspace$ is a solution to the martingale problem associated with \eqref{eq:langevin_limit}.
\end{proposition}

\begin{proof}
The proof is postponed to \Cref{sec:reduction_martingale_problem}.
\end{proof}

\begin{theorem}
\label{theo:diffusion_limit_RMW}
Assume \Cref{assum:diff:quadratic} and \Cref{assum:Vdot} hold. Assume also that \eqref{eq:langevin_limit} has a unique weak solution. Then, $\defEns{(Y_{t,1}^{\dim})_{t\geq 0}, \dim \in \Nset^*}$ converges weakly to the solution $\parenthese{Y_t}_{
t \geq 0}$ of the Langevin equation defined
by \eqref{eq:langevin_limit}. Furthermore, $h(\ell)$ is maximized at the unique value of $\ell$ for which $a
(\ell) = 0.234$, where $a$ is defined in \Cref{theo:result_acceptance_rate_RWM}.
\end{theorem}

\begin{proof}
The proof is postponed to \Cref{sec:proof:weaklimit}.
\end{proof}

\begin{example}[Bayesian Lasso]
We apply the results obtained above to a target density $\pi$ on $\rset$ given by $x \mapsto \rme^{-V(x)}/ \int_{\rset} \rme^{-V(y)} \rmd y$ where $V$ is given by
$$
V: x \mapsto \smoothPotLasso(x) + \coeffLasso \abs{x} \eqsp,
$$
where $\coeffLasso \geq 0$ and $\smoothPotLasso$ is twice continuously differentiable with bounded second derivative. Furthermore, $\int_{\rset}\abs{x}^6 \rme^{-V(x)}\rmd x < \plusinfty$. Define  $\dot{V}: x \mapsto \smoothPotLasso'(x) + \coeffLasso\,  \sign(x)$, with $\sign(x) = -1$ if $x \leq 0$ and $\sign(x) = 1$ otherwise. We first check that \Cref{assum:diff:quadratic}\ref{hyp:mean_square_deriv2} holds.
Note that for all $x,y \in \rset$,
\begin{equation}
\label{eq:dev_valeur_absolue}
\abs{\abs{x+y} -\abs{x} - \sign(x)y}
\leq 2|y| \1_{\rset_+}(\abs{y} -\abs{x} ) \eqsp,
\end{equation}
which implies that, for any $p \geq 1$, there exists $C_p$ such that
\begin{align*}
\norm{V(\cdot+\theta)-V(\cdot) - \theta\Vdot(\cdot)}_{\pdf,p} &\leq \norm{\smoothPotLasso(\cdot+\theta)-\smoothPotLasso(\cdot) -\theta \smoothPotLasso'(\cdot)}_{\pdf,p}+ \coeffLasso \norm{\abs{\cdot+\theta}-\abs{\cdot} - \theta\,\sign(\cdot) }_{\pdf,p} \\
& \leq \norm{U''}_{\infty}\,\theta^2 + 2\,\left|\theta\right| \coeffLasso
\{ \pdf(\ccint{-\theta,\theta}) \}^{1/p}  \leq
C \abs{\theta}^{p+1/p} \vee \abs{\theta}^{2}\eqsp.
\end{align*}
Assumptions \Cref{assum:diff:quadratic}\ref{assum:X6} and \Cref{assum:Vdot} are easy to check.
The uniqueness in law of  \eqref{eq:langevin_limit} is established in \cite[Theorem 4.5 (i)]{cherny:engelbert:2005}.
Therefore, \Cref{theo:diffusion_limit_RMW} can be applied.
\end{example}


\section{Target density supported on an interval of $\rset$}
\label{sec:I}
We now extend  our results  to densities supported by a open interval $\I  \subset \rset$ :
\[
\unitarget(x)\propto \exp(-V(x))\1_{\I}(x)\eqsp,
\]
where $V: \I \to \rset$ is a measurable function. Note that by convention $V(x) = -\infty$ for all $x\notin \I$. Denote by $\Iclosed$ the closure of $\I$ in $\rset$. The results of \Cref{sec:R} cannot be directly used in such a case, as $\unitarget$ is no longer positive on $\Rset$. Consider the following assumption.
\begin{assumptionG}
\label{assum:diff:quadratic:G}
There exists a measurable function $\Vdot:\I\to\Rset$ and $\constSet >1$ such that:
\begin{enumerate}[label=(\roman*)]
\item
\label{hyp:mean_square_deriv2:G}
There exist $p >4$, $C>0$ and $\beta>1$ such that for all $\theta\in\rset$,
\[
\norm{\{V(\cdot+\theta)-V(\cdot)\}\1_{\I}(\cdot + \constSet \theta)\1_{\I}(\cdot+(1-\constSet) \theta)-\theta\Vdot(\cdot)}_{\pdf,p} \le C|\theta|^{\beta}\eqsp,
\]
with the convention $0 \times \infty = 0$.
\item
\label{assum:X6:G}
The function $\Vdot$ satisfies $\Vert \Vdot \Vert_{\pi,6}<+\infty$.
\item
\label{assum:proba:G}
There exist $\gamma\ge6$ and $C>0$ such that, for all $\theta\in\Rset$,
$$
\int_{\rset} \1_{\I^c}(x+\theta)\pi(x)\rmd x\le C |\theta|^{\gamma} \eqsp.
$$
\end{enumerate}
\end{assumptionG}

As an important consequence of
\Cref{assum:diff:quadratic:G}\ref{assum:proba:G}, if $X$ is
distributed according to $\pi$ and is independent of the standard
random variable $Z$, there exists a constant $C$ such that
\begin{equation}
\label{eq:application-hyp}
\PP \left( X + \ell \dim^{-1/2} Z \in \I^c\right) \le C d^{-\gamma/2}\eqsp.
\end{equation}
\begin{theorem}
\label{theo:result_acceptance_rate_RWM:G}
Assume \Cref{assum:diff:quadratic:G} holds. Then,
$\lim_{\dim \to \plusinfty } \PP\left[\setAccept_1^\dim \right] = a(\ell)$, where $a(\ell) = 2 \Phi(-\sqrt{I} \ell /2)$.
\end{theorem}
\begin{proof}
The proof is postponed to \Cref{proof:lem:approx_ratio:G}.
\end{proof}
We now established the weak convergence of the sequence $\{(Y_{t,1}^{\dim})_{t\ge 0}, \dim\in \Nset^{\star}\}$, following the same steps as for the proof of \Cref{theo:diffusion_limit_RMW}. Denote for all $d\geq 1$, $\mu_d$ the law of the process $(Y_{t,1}^{\dim})_{t\ge 0}$.
\begin{proposition}
\label{prop:tight:G}
Assume \Cref{assum:diff:quadratic:G} holds. Then, the sequence $\left(\mu_d\right)_{d\ge 1}$ is tight in $\Wienerspace$.
\end{proposition}
\begin{proof}
The proof is postponed to \Cref{sec:proof:tightness:G}.
\end{proof}
Contrary to the case where $\pi$ is positive on $\rset$, we do not assume that $\Vdot$ is bounded on all compact sets of $\rset$. Therefore, we consider the local martingale problem associated with \eqref{eq:langevin_limit}:  with the notations of \Cref{sec:R}, a probability measure
$\mu$ on $\Wienerspace$ is said to solve the local martingale problem associated with \eqref{eq:langevin_limit} if
the pushforward of $\mu$ by $W_0$ is  $\pdf$ and if for all $\psi \in C^{\infty}(\rset,\rset)$,
the process
$$\left(\psi(W_t) - \psi(W_0) - \int_{0}^t \Lrm \psi(W_u) \rmd u\right)_{ t \geq 0} $$
is a local martingale with respect to $\mu$ and the filtration
$(\canonicalFiltration_t)_{t \geq 0}$.  By \cite[Theorem
1.27]{cherny:engelbert:2005}, any solution to the local martingale
problem associated with \eqref{eq:langevin_limit} coincides with the
law of a solution to the SDE \eqref{eq:langevin_limit} and
conversely. If \eqref{eq:langevin_limit} admits a  unique solution in law,  this law is the unique solution
to the local martingale problem associated with \eqref{eq:langevin_limit}.  We first prove that any
limit point $\mu$ of $(\mu_d)_{d\ge 1}$ is a solution to the local martingale
problem associated with \eqref{eq:langevin_limit}.

\begin{assumptionG}
\label{assum:Vdot:G}
The function $\Vdot$ is continuous on $\I$ except on a null-set $\setDisconDotV$, \wrt~the
Lebesgue measure, and is bounded on all compact sets of $\I$.
\end{assumptionG}
This condition does not preclude that $\Vdot$ is unbounded at the boundary of $\I$.
\begin{proposition}
\label{propo:reduction_martingale_problem:G}
Assume \Cref{assum:diff:quadratic:G} and \Cref{assum:Vdot:G} hold. Assume also that for all $\phi \in C_c^{\infty}(\I,\rset)$, $m \in \nset^*$, $g : \rset^m \to \rset$ bounded and continuous, and $0 \leq t_1 \leq \dots \leq t_m \leq s \leq t$:
 \begin{equation}
 \label{eq:reduction_martingale_problem:G}
 \lim_{d \to \plusinfty} \PE^{\mu_d}\left[ \left(\phi\left(W_t\right) - \phi\left( W_s \right) - \int_s^t \Lrm \phi \left(W_u \right) \rmd u  \right) g\left(W_{t_1},\dots,W_{t_m} \right) \right] =0 \eqsp.
\end{equation}
Then, every limit point of the sequence of probability measures $(\mu_{\dim})_{\dim\ge 1}$ on $\Wienerspace$ is a solution to the local martingale problem associated with \eqref{eq:langevin_limit}.
\end{proposition}
\begin{proof}
The proof is postponed to \Cref{sec:reduction_martingale_problem:G}.
\end{proof}

\begin{theorem}
\label{theo:diffusion_limit_RMW:G}
Assume \Cref{assum:diff:quadratic:G} and \Cref{assum:Vdot:G} hold. Assume also that \eqref{eq:langevin_limit} has a unique weak solution. Then, $\defEns{(Y_{t,1}^{\dim})_{t\geq 0}, \dim \in \Nset^*}$ converges weakly to the solution $\parenthese{Y_t}_{
t \geq 0}$ of the Langevin equation defined
by \eqref{eq:langevin_limit}. Furthermore, $h(\ell)$ is maximized at the unique value of $\ell$ for which $a
(\ell) = 0.234$, where $a$ is defined in \Cref{theo:result_acceptance_rate_RWM}.
\end{theorem}

\begin{proof}
The proof is along the same lines as the proof of \Cref{theo:diffusion_limit_RMW} and is postponed to \Cref{sec:proof:weaklimit:G}. 
\end{proof}
The conditions for uniqueness in law of singular one-dimensional stochastic differential equations are given in \cite{cherny:engelbert:2005}. These conditions are rather involved and difficult to summarize in full generality. We rather illustrate \Cref{theo:diffusion_limit_RMW:G} by two examples.
\begin{example}[Application to the Gamma distribution]
\label{ex:gamma}
Define the class of the generalized Gamma distributions as the family of densities on $\rset$ given by
\[
\pigamma : x \mapsto x^{\rma_1-1}\exp(-x^{\rma_2 })\1_{\rset_+^{\star}}(x)\Big/\int_{\rset_+^{\star}} y^{\rma_1-1}\exp(-y^{\rma_2})\rmd y \eqsp,
\]
with two parameters $\rma_1>6$ and $\rma_2 >0$.
Note that in this case $\I = \rset_+^{\star}$, for all $x\in\I$,
$\Vgamma:x\mapsto x^{\rma_2}- (\rma_1-1)\log x$ and  $\Vgammadot:x\mapsto \rma_2 x ^{\rma_2-1} - (\rma_1-1)/x$.
We check that \Cref{assum:diff:quadratic:G} holds with $\constSet = 3/2$.  First, we show that
\Cref{assum:diff:quadratic:G}\ref{hyp:mean_square_deriv2:G} holds with 
$p= 5 $. Write for all $\theta\in\rset$ and $x\in\I$,
$$
\left\{\Vgamma(x+\theta)-\Vgamma(x)\right\}\1_{\I}(x+(1-r)\theta)\1_{\I}(x+r\theta)-\theta\Vgammadot(x) = \mathcal{E}_1 + \mathcal{E}_2 + \mathcal{E}_3 \eqsp,
$$
 where
\begin{align*}
\mathcal{E}_1 & = \theta \Vgammadot(x) \left\{\1_{\I}(x-\theta/2)\1_{\I}(x+3\theta/2)-1 \right\}\eqsp,\\
\mathcal{E}_2 & = (1-\rma_1)\left\{\log (1+\theta/x)-\theta/x\right\}\1_{\I}(x-\theta/2)\1_{\I}(x+3\theta/2)\eqsp,\\
\mathcal{E}_3 & = ((x+\theta)^{\rma_2} - x^{\rma_2} - \rma_2\theta  x^{\rma_2-1})\1_{\I}(x-\theta/2)\1_{\I}(x+3\theta/2)\eqsp.
\end{align*}
It is enough to prove that there exists $q >5$ such that for all $i\in\{1,2,3\}$, $\int_{\I} \left|\mathcal{E}_i\right|^5\pigamma(x)\rmd x \le C|\theta|^q$. The result is proved for $\theta<0$ (the proof for $\theta>0$ follows the same lines). For all $\theta\in\rset$ using $\rma_1 >6$,
\begin{align}
\nonumber
\int_{\rset^*_+} \left|\mathcal{E}_1\right|^5\pigamma(x)\rmd x 
&\le C |\theta|^{5}\int_{0}^{3|\theta|/2}\left\{1/x^5 + x^{5(\rma_2-1)}\right\}x^{\rma_1-1}\mathrm{e}^{-x^{\rma_2}}\rmd x\eqsp, \\
\nonumber
& \le C |\theta|^{\rma_1} \parenthese{\int_{0}^{3/2}x^{\rma_1-6}\mathrm{e}^{-(|\theta|x)^{\rma_2}}\rmd x + |\theta|^{5\rma_2}\int_{0}^{3/2}x^{5(\rma_2-1)+\rma_1-1}\mathrm{e}^{-(|\theta|x)^{\rma_2}}\rmd x }\eqsp, \nonumber\\
\label{eq:item:verif_gamma_1}
&\le C (|\theta|^{\rma_1}+|\theta|^{5\rma_2+\rma_1}) \eqsp.
\end{align}
On the other hand, as for all $x>-1$, $x/(x+1)\le\log(1+x)\le x$, for all $\theta<0$, and $x\ge 3|\theta|/2$,
\[
\left|\log(1+\theta/x)-\theta/x\right|\le \frac{|\theta|^2}{x^2(1+\theta/x)}\le 3 |\theta|^2/x^2\eqsp,
\]
where the last inequality come from $|\theta|/x\le 2/3$. Then, it yields
\begin{align}
\int_{\rset^*_+} \left|\mathcal{E}_2(x)\right|^5\pigamma(x)\rmd x &\le C|\theta|^{10} \parenthese{\int_{3|\theta|/2}^{1} x^{\rma_1-11}\mathrm{e}^{-x^{\rma_2}}\rmd x + \int_{1}^{+\infty} x^{\rma_1-11}\mathrm{e}^{-x^{\rma_2}}\rmd x}\eqsp,\nonumber\\
&\le C(|\theta|^{\rma_1}+|\theta|^{10})\eqsp.
\label{eq:item:verif_gamma_2}
\end{align}
For the last term, for all $\theta<0$ and all $x\ge 3|\theta|/2$, using a Taylor expansion of $x\mapsto x^{\rma_2}$, there exists $\zeta \in [x+\theta,x]$ such that
\[
\left|(x+\theta)^{\rma_2} - x^{\rma_2} - \rma_2\theta  x^{\rma_2-1}\right|\le C |\theta|^{2}|\zeta|^{\rma_2-2} \le C|\theta|^{2}\left|x\right|^{\rma_2-2}\eqsp.
\]
Then,
\begin{equation}
\int_{\rset^*_+} \left|\mathcal{E}_3(x)\right|^5 \pigamma(x)\rmd x \le C|\theta|^{10}\int_{3|\theta|/2}^{+\infty}x^{5(\rma_2-2) + \rma_1-1}\mathrm{e}^{-x^{\rma_2}}\rmd x
\leq C(|\theta|^{5\rma_2 + \rma_1} + |\theta|^{10}) \eqsp.
\label{eq:item:verif_gamma_3}
\end{equation}
Combining \eqref{eq:item:verif_gamma_1}, \eqref{eq:item:verif_gamma_2},\eqref{eq:item:verif_gamma_3} and using that $\rma_1 > 6$ concludes the proof of \Cref{assum:diff:quadratic:G}\ref{hyp:mean_square_deriv2:G} for $p=5$. The proof of \Cref{assum:diff:quadratic:G}\ref{assum:X6:G} follows from 
\begin{align*}
\int_{\rset^*_+}|\Vgammadot(x)|^6\pigamma(x) \rmd x
& \leq C\left(\int_{\rset^*_+}x^{\rma_1-1+6(\rma_2-1)}\mathrm{e}^{-x^{\rma_2}}\rmd x+\int_{\rset^*_+}x^{\rma_1-7}\mathrm{e}^{-x^{\rma_2}}\rmd x\right) < \infty
\end{align*}
and \Cref{assum:diff:quadratic:G}\ref{assum:proba:G} follows from $\int_{\rset} \1_{\I^c}(x+\theta) \pigamma(x)\rmd x \le C |\theta|^{\rma_1}$. Now consider the Langevin equation associated with $\pigamma$ given by
$ \rmd Y_t = -\Vgammadot(Y_t) \rmd t + \sqrt{2} \rmd B_t $
with initial distribution $\pigamma$.  This stochastic differential
equation has $0$ as singular point, which has right type $3$ according
to the terminology of \cite{cherny:engelbert:2005}. On the other hand
$\infty$ has type $A$ and the existence and uniqueness in law for
the SDE follows from \cite[Theorem 4.6
(viii)]{cherny:engelbert:2005}. Since \Cref{assum:Vdot:G} is
straightforward, \Cref{theo:diffusion_limit_RMW:G} can be applied.
\end{example}

\begin{example}[Application to the Beta distribution]
Consider now the case of the Beta distributions $\pibeta$ with
density $x\mapsto x^{a_1-1}(1-x)^{a_2-1} \1_{\ooint{0,1}}(x)$ with
$a_1,a_2>6$. Here $\I = \ooint{0,1}$ and the log-density $\Vbeta$ and its derivative on $\I$ are defined by
$\Vbeta(x)= - (a_1-1)\log x - (a_2-1) \log(1-x)$ and $\Vbetadot(x)=  - (a_1-1)/x - (a_2-1)/(1-x)$. 
Along the same lines as above, $\pibeta$ satisfies \Cref{assum:diff:quadratic:G} and \Cref{assum:Vdot:G}. Hence
\Cref{theo:result_acceptance_rate_RWM:G} can be applied if we establish the
uniqueness in law for the Langevin equation associated with $\pibeta$
defined by $  \rmd Y_t = -\Vbetadot(Y_t) \rmd t + \sqrt{2} \rmd B_t $
with initial distribution $\pibeta$. In the terminology
of \cite{cherny:engelbert:2005}, $0$ has right type $3$ and $1$ has
left type $3$. Therefore by \cite[Theorem 2.16 (i),
(ii)]{cherny:engelbert:2005}, the SDE has a global unique weak solution.
To illustrate our findings, consider the Beta
distribution with parameters $a_1=10$ and $a_2=10$. Define the expected
square distance by $\ESJD^d(\ell) = \expe{\norm{X_1^d-X_0^d}^2}$
 where $X_0^d$ has distribution $\pibeta^d$ and $X_1^d$ is the first
 iterate of the Markov chain defined by the Random Walk Metropolis
 algorithm given in \eqref{eq:proposaliid}. By
 \Cref{theo:result_acceptance_rate_RWM:G} and
 \Cref{theo:diffusion_limit_RMW:G}, we have $\lim_{d \to \plusinfty}
 \ESJD^d(\ell) = h(\ell) = \ell^2 a(\ell)$.  \Cref{fig:ESJD_beta} displays an  empirical estimation for the $\ESJD^d$ for dimensions
 $d=10,50,100$ as a function of the empirical mean acceptance rate. We
 can observe that as expected, the $\ESJD^d$ converges to some limit
 function as $d$ goes infinity, and this function has a maximum for a
 mean acceptance probability around $0.23$.
 \begin{figure}[h]
   \centering
   \includegraphics[scale=0.5]{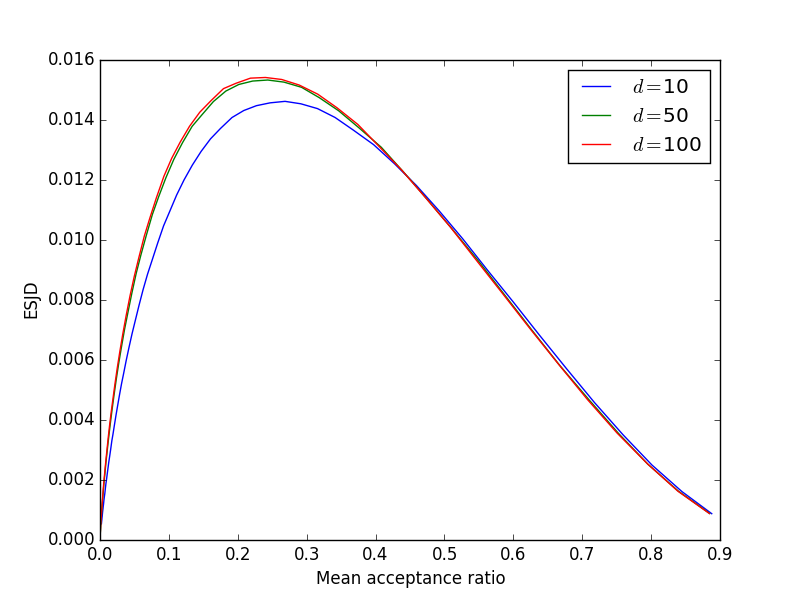}
   \caption{Expected square jumped distance for the beta distribution with parameters $a_1=10$ and $a_2=10$ as a function of the mean acceptance rate for $d=10,50,100$.}
   \label{fig:ESJD_beta}
 \end{figure}
 \end{example}


\section{Proofs of \Cref{sec:R}}
\label{sec:proofs}
For any real random variable $Y$ and any $p \geq 1$, let $\norm{Y}_{p} \eqdef \PE[|Y|^p]^{1/p}$.
\subsection{Proof of \Cref{lem:DQM}}
\label{sec:proof-DQM}
Let $\Delta_{\theta}V(x) = V(x)-V(x+\theta)$. By definition of $\sqpdf_\theta$ and $\unitarget$,
\[
 \left(\sqpdf_\theta(x)-\sqpdf_0(x) + \theta \Vdot(x)\sqpdf_0(x)/2\right)^2 \le 2\left\{A_{\theta}(x) + B_{\theta}(x)\right\}\unitarget(x)\eqsp,
\]
where
\begin{align*}
&A_{\theta}(x)= \left(\exp(\Delta_{\theta}V(x)/2) - 1 - \Delta_{\theta}V(x)/2\right)^2 \eqsp, \\
&B_{\theta}(x)= \left(\Delta_{\theta}V(x)+ \theta \Vdot(x)\right)^2/4.
\end{align*}
By \Cref{assum:diff:quadratic}\ref{hyp:mean_square_deriv2}, $\norm{B_{\theta}}_{\pdf,p} \le C|\theta|^{\beta}$. For $A_{\theta}$, note that for all $x\in\rset$, $(\exp(x)-1-x)^2 \le 2x^4(\exp(2x)+1)$. Then,
\begin{align*}
\int_{\rset} A_{\theta}(x) \unitarget(x)\rmd x &\le C \int_{\rset} \Delta_{\theta}V(x)^4 \left(1+\rme^{\Delta_{\theta}V(x)}\right) \unitarget(x)\rmd x \\
&\le  C \int_{\rset} \left(\Delta_{\theta}V(x)^4 + \Delta_{-\theta}V(x)^4\right) \unitarget(x)\rmd x \eqsp.
\end{align*}
The proof is completed writing (the same inequality holds for $\Delta_{-\theta}V$):
\[
\int_{\rset} \Delta_{\theta}V(x)^4 \unitarget(x)\rmd x \le C\left[\int_{\rset} \left(\Delta_{\theta}V(x)-\theta\Vdot(x)\right)^4 \unitarget(x)\rmd x + \theta^4\int_{\rset} \Vdot^4(x) \unitarget(x)\rmd x\right]
\]
and using \Cref{assum:diff:quadratic}\ref{hyp:mean_square_deriv2}-\ref{assum:X6}.

\subsection{Proof of \Cref{lem:approx_ratio}}
\label{proof:lem:approx_ratio}
Define
\begin{equation}
\label{eq:def_R}
R(x) = \int_0^x \frac{(x-u)^2}{(1+u)^3} \rmd u \eqsp.
\end{equation}
$R$ is the remainder term of the Taylor expansion of $x\mapsto \log(1+x)$:
\begin{equation}
  \label{eq:dev_taylor_log}
\log(1+x) = x-x^2/2+R(x) \eqsp.
\end{equation}
We preface the proof by the following Lemma.
\begin{lemma}
\label{lem:integrated-DQM}
Assume \Cref{assum:diff:quadratic} holds. Then,
if $X$ is a random variable distributed according to $\pi$ and $Z$ is a standard Gaussian
random variable independent of $X$,
\begin{enumerate}[label=(\roman*)]
\item \label{lem:integrated-DQM-L2} $\lim_{d \to \plusinfty} d \ \left\|\zeta^{\dim}(X,Z) +\ell Z\Vdot(X)/( 2\sqrt{\dim}) \right\|_2^2 = 0$.
\item \label{lem:integrated-DQM-Lp} $\lim_{d \to \plusinfty} \sqrt{d}\norm{V(X)-V(X+\ell Z /\sqrt{d})+\ell Z\Vdot(X)/\sqrt{d}}_p = 0.$
\item \label{lem:integrated-DQM-remainder} $\lim_{\dim \to \infty} \dim \norm{ R\left(\zeta^{\dim}(X,Z)\right)}_1 = 0$,
\end{enumerate}
where $\zeta^{\dim}$ is given by \eqref{eq:def:zeta}.
\end{lemma}
\begin{proof}
Using the definitions \eqref{eq:def_xitheta} and \eqref{eq:def:zeta} of $\zeta^{\dim}$ and $\sqpdf_\theta$ ,
\begin{equation}
\label{eq:relationxi_zeta}
\zeta^{\dim}(x,z)=  \sqpdf_{\ell z\dim^{-1/2}}(x)/\sqpdf_0(x)-1 \eqsp.
\end{equation}

\begin{enumerate}[label=(\roman*), wide=0pt, labelindent=\parindent]
\item The proof follows from \Cref{lem:DQM} using that $\beta> 1$:
\begin{equation*}
\left\|\zeta^{\dim}(X,Z) +\ell Z\Vdot(X)/( 2\sqrt{\dim})\right\|_2^2 
\leq C \ell^{2\beta}\dim^{-\beta}\PE\left[|Z|^{2\beta}\right]\eqsp.
\end{equation*}

\item Using \Cref{assum:diff:quadratic}\ref{hyp:mean_square_deriv2}, we get that
\[
\norm{V(X)-V(X+\ell Z /\sqrt{d})+\ell Z\Vdot(X)/\sqrt{d}}_p^p
\le C\ell^{\beta p}\dim^{-\beta p/2}\PE\left[\left|Z\right|^{\beta p}\right]
\]
and the proof follows since $\beta> 1$.
\item
Note that for all $x >0$, $u \in \ccint{0,x}$, $\vert
(x-u)(1+u)^{-1} \vert \leq \abs{x}$, and the same inequality holds for $x \in
\ocint{-1,0}$ and $u \in \ccint{x,0}$. Then by \eqref{eq:def_R} and
\eqref{eq:dev_taylor_log}, for all $x>-1$, $\abs{R(x)} \le x^2\left|\log(1+x)\right|$.

Then by \eqref{eq:relationxi_zeta}, setting $\Psi_d(x,z)= R(\zeta^\dim(x,z))$
\begin{equation*}
\left| \Psi_d(x,z)\right|
\le \left(\sqpdf_{\ell z\dim^{-1/2}}(x)/\sqpdf_0(x)-1\right)^2\left|V(x+\ell z\dim^{-1/2}) -V(x)\right|/2\eqsp.
\end{equation*}
Since for all $x\in\rset$, $|\exp(x)-1|\le |x|(\exp(x)+1)$, this yields,
\[
\left|\Psi_d(x,z) \right| \le 4^{-1}\left|V(x+\ell z\dim^{-1/2}) -V(x)\right|^3\left(\exp\left(V(x)-V(x+\ell z\dim^{-1/2})\right)+1\right)\eqsp,
\]
which implies that
\[
 \int_{\rset} \left|\Psi_d(x,z)\right| \pi(x)\rmd x \le 4^{-1}\int_{\rset} \left|V(x+\ell z\dim^{-1/2}) -V(x)\right|^3 \{ \pi(x) + \pi(x+\ell z\dim^{-1/2}) \} \rmd x\eqsp.
\]
By H\"older's inequality and using \Cref{assum:diff:quadratic}\ref{hyp:mean_square_deriv2},
\[
\int_{\rset} \left|\Psi_d(x,z)\right| \pi(x)\rmd x \le C\left(\left|\ell z\dim^{-1/2}\right|^3 \left(\int_{\rset} \left|\Vdot(x)\right|^4\pi(x)\rmd x\right)^{3/4}+ \left|\ell z\dim^{-1/2}\right|^{3\beta}\right)\eqsp.
\]
The proof follows from \Cref{assum:diff:quadratic}\ref{assum:X6} since $\beta>1$.
\end{enumerate}
\end{proof}
For all $d \geq 1$, let $X^d$ be distributed according to $\pi^d$, and $Z^d$ be
$d$-dimensional Gaussian random variable independent of $X^d$, set
$$
\Jrm^{\dim}= \norm{\sum_{i=2}^{\dim}
\left\{ \Delta V^{\dim}_{i} - b^\dim(X_i^\dim,Z_i^\dim) \right\}}_1 \eqsp,
$$
where $\Delta V^{\dim}_{i}$ and $b^\dim$ are defined in \eqref{eq:Delta_V} and \eqref{eq:bdi}, respectively.
\begin{lemma}
\label{lem:interm_prop_approx_ratio}
$\lim_{\dim\to\plusinfty}\Jrm^{\dim}= 0$.
\end{lemma}
  \begin{proof}
Noting that $\Delta V^{\dim}_{i} = 2\log\left(1+\zeta^{\dim}\left(X_{i}^{\dim},Z_i^{\dim}\right)\right)$ and using \eqref{eq:dev_taylor_log}, we get
\begin{multline*}
\Jrm^{\dim} \leq \sum_{i=1}^3 \Jrm_i^{\dim} = \left\|\sum_{i=2}^{\dim}2\zeta^{\dim}\left(X_{i}^{\dim},Z_i^{\dim}\right) +\frac{\ell
  Z_{i}^{\dim}}{\sqrt{\dim}}\Vdot(X_{i}^{\dim})-\PE\left[2\zeta^{\dim}(X_{i}^{\dim},Z_{i}^{\dim})\right]\right\|_1 \\ + \left\|\sum_{i=2}^{\dim}\zeta^{\dim}\left(X_{i}^{\dim},Z_i^{\dim}\right)^2-\frac{\ell^2 }{4\dim}\Vdot^2(X_{i}^{\dim})\right\|_1
  +2 \left\|\sum_{i=2}^{\dim}R\left(\zeta^{\dim}\left(X_{i}^{\dim},Z_i^{\dim}\right)\right)\right\|_1\eqsp,
\end{multline*}
where $R$ is defined by \eqref{eq:def_R}. By \Cref{lem:integrated-DQM}\ref{lem:integrated-DQM-L2}, the first term goes to $0$ as $\dim$ goes to $+\infty$ since
\[
  \Jrm_1^{\dim} \le \sqrt{\dim}\,\left\|2\zeta^{\dim}\left(X_{1}^{\dim},Z_1^{\dim}\right) +\frac{\ell
  Z_{1}^{\dim}}{\sqrt{\dim}}\Vdot(X_{1}^{\dim})\right\|_2\eqsp.
\]
Consider now $\Jrm_2^{\dim}$. We use the following decomposition for all $2\le i \le \dim$,
\begin{multline*}
\zeta^{\dim} (X_i^{d},Z_i^{d})^2 - \frac{\ell^2
}{4\dim}\Vdot^2(X_i^{d}) = \left(\zeta^{\dim}(X_i^{d},Z_i^{d}) + \frac{\ell
}{2\sqrt{\dim}}Z_i^{\dim}\Vdot(X_i^{d})\right)^2 \\
-\frac{\ell}{\sqrt{\dim}}Z_i^{\dim}\Vdot(X_i^{d})\left(\zeta^{\dim}(X_i^{d},Z_i^{d})+\frac{\ell}{2\sqrt{\dim}}Z_i^{\dim}\Vdot(X_i^{d})
\right) + \frac{\ell^2}{4\dim}\left\{(Z_i^{\dim})^2-1\right\}\Vdot^2(X_i^{d}) \eqsp.
\end{multline*}
Then,
\begin{multline*}
\Jrm^{\dim}_2 \le \dim\,\left\|\zeta^{\dim}(X_1^{d},Z_1^{d})
+ \frac{\ell
}{2\sqrt{\dim}}Z_1^{\dim}\Vdot(X_1^{d})\right\|_2^2  + \frac{\ell^2}{4\dim}\,\left\|\sum_{i=2}^{\dim}\Vdot^2(X_i^{d})\left\{(Z_i^{\dim})^2-1\right\}\right\|_1\\
+ \ell\sqrt{\dim}\,\left\|\Vdot(X_1^{d})Z_1^{\dim}\left(\zeta^{\dim}(X_1^{d},Z_1^{d}) + \frac{\ell
}{2\sqrt{\dim}}Z_1^{\dim}\Vdot(X_1^{d})\right)\right\|_1
\eqsp.
\end{multline*}
Using \Cref{assum:diff:quadratic}\ref{assum:X6}, \Cref{lem:integrated-DQM}\ref{lem:integrated-DQM-L2} and the Cauchy-Schwarz
inequality show that the first  and the last term converge to zero. For the second term note that $\PE\left[(Z_i^{\dim})^2-1\right] = 0$ so that
\begin{equation*}
\dim^{-1} \norm{\sum_{i=2}^{\dim}\Vdot^2(X_i^{d})\left\{(Z_i^{\dim})^2-1\right\}}_1
\leq \dim^{-1/2} \,\mathrm{Var}\left[\Vdot^2(X_1^{d})\left\{(Z_1^{\dim})^2-1\right\}\right]^{1/2} \to 0\eqsp.
\end{equation*}
Finally, $\lim_{\dim \to \infty} \Jrm_3^{\dim}= 0$ by  \eqref{eq:dev_taylor_log} and \Cref{lem:integrated-DQM}\ref{lem:integrated-DQM-remainder}.
\end{proof}

\begin{proof}[Proof of \Cref{lem:approx_ratio}]
Let $q >0$ and $\Lambda^d = - \ell \dim^{-1/2} Z_{1}^{d}\Vdot(X_{1}^{d}) + \sum_{i=2}^d \Delta V^{\dim}_{i}$. By the triangle inequality,
$\Erm^{\dim}(q)\leq \Erm^{\dim}_1(q) + \Erm^{\dim}_2(q)$ where
\begin{align*}
\Erm^{\dim}_1(q) &= \PE\left[ \left(Z_{1}^d\right)^q \left|  1 \wedge \exp\left\{\sum_{i=1}^d \Delta V^{\dim}_{i}\right\} - 1 \wedge \exp\left\{\Lambda^d\right\} \right| \right]\eqsp,\\
\Erm^{\dim}_2(q) &= \PE\left[ \left(Z_{1}^d\right)^q  \left|1 \wedge \exp\left\{\Lambda^d\right\}   - 1 \wedge \exp\left\{\upsilon^d\right\} \right| \right] \eqsp.
\end{align*}
Since $t \mapsto 1 \wedge \rme^t$ is $1$-Lipschitz, by the Cauchy-Schwarz inequality we get
\[
\Erm^{\dim}_1(q) \leq \norm{Z_1^{\dim}}_{2q}^q \norm{\Delta V^{\dim}_1 + \ell \dim^{-1/2} Z_{1}^{d}\Vdot(X_{1}^{d})}_2  \eqsp.
\]
By \Cref{lem:integrated-DQM}\ref{lem:integrated-DQM-Lp}, $\Erm^{\dim}_1(q)$ goes to $0$ as $d$ goes to $\plusinfty$.
Consider now $\Erm^{\dim}_2(q)$.  Using again that $t \mapsto 1 \wedge \rme^t$ is $1$-Lipschitz and \Cref{lem:interm_prop_approx_ratio}, 
$\Erm^{\dim}_2(q)$ goes to $0$.
\end{proof}

\subsection{Proof of \Cref{theo:result_acceptance_rate_RWM}}
\label{proof:theo:result_acceptance_rate_RWM}
Following \cite{jourdain:lelievre:miasojedow:2015}, we introduce the function $\G$ defined on $\overline{\Rset}_+\times \Rset$ by:
\begin{equation}
\label{eq:defG}
\G(a,b) = \left\{
    \begin{array}{ll}
        \mathrm{exp}\left(\frac{a-b}{2}\right)\Phi\left(\frac{b}{2\sqrt{a}}-\sqrt{a}\right) &
        \mbox{if}\; a\in(0,+\infty)\eqsp, \\
        0 & \mbox{if}\; a=+\infty\eqsp,\\
        \mathrm{exp}\left(-\frac{b}{2}\right)\1_{\left\{b>0\right\}} & \mbox{if}\; a=0\eqsp,
    \end{array}
\right.
\end{equation}
where $\Phi$ is the cumulative distribution function of a standard normal variable, and $\Gamma$:
\begin{equation}
\label{eq:defGamma}
\Gamma(a,b) = \left\{
    \begin{array}{ll}
        \Phi\left(-\frac{b}{2\sqrt{a}}\right) + \mathrm{exp}\left(\frac{a-b}{2}\right)\Phi\left(\frac{b}{2\sqrt{a}}-\sqrt{a}\right) &
        \mbox{if}\; a\in(0,+\infty)\eqsp, \\
        \frac{1}{2} & \mbox{if}\; a=+\infty\eqsp,\\
        \mathrm{exp}\left(-\frac{b_+}{2}\right) & \mbox{if}\; a=0\eqsp.
    \end{array}
\right.
\end{equation}
Note that $\mathcal{G}$ and $\Gamma$ are bounded on $\overline{\Rset}_+\times \Rset$. $\mathcal{G}$ and $\Gamma$ are used throughout \Cref{sec:proofs}.

\begin{lemma}
\label{lem:mean:zeta}
Assume \Cref{assum:diff:quadratic} holds. For
all $d \in \nset^*$, let
$X^{\dim}$ be a random variable distributed according to $\target^{\dim}$ and $Z^{\dim}$
be a standard Gaussian random variable in $\rset^d$, independent of $X$. Then,
\[
\lim_{d \to \plusinfty}\dim\,\PE\left[2\zeta^{\dim}(X^{\dim}_1,Z^{\dim}_1)\right] = -\frac{\ell^2}{4}I\eqsp,
\]
where $I$ is defined in \eqref{eq:translation-score} and $\zeta^{\dim}$ in \eqref{eq:def:zeta}.
\end{lemma}
\begin{proof}
By \eqref{eq:def:zeta},
\begin{multline*}
\dim\,\PE\left[2\zeta^{\dim}(X^{\dim}_1,Z^{\dim}_1)\right] =2\dim\PE\left[\int_{\rset} \sqrt{\pdf\left(x+ \ell \dim^{-1/2} Z^{\dim}_1\right)}\sqrt{\pdf\left(x\right)}\rmd x-1\right]\eqsp,\\
=  -\dim\PE\left[\int_{\rset} \left(\sqrt{\pdf\left(x+ \ell\dim^{-1/2}\, Z^{\dim}_1\right)}-\sqrt{\pdf\left(x\right)}\right)^2\rmd x\right] = -\dim\,\PE\left[\{\zeta^{\dim}(X^{\dim}_1,Z^{\dim}_1)\}^2\right]\eqsp.
\end{multline*}
The proof is then completed by \Cref{lem:integrated-DQM}\ref{lem:integrated-DQM-L2}.
\end{proof}

\begin{proof}[Proof of \Cref{theo:result_acceptance_rate_RWM}]
By definition of $\setAccept_1^\dim$, see \eqref{eq:setAccept},
\[
\PP\left[\setAccept_1^\dim \right] = \PE\left[1 \wedge \exp\left\{\sum_{i=1}^d \Delta V^{\dim}_{i}\right\}\right]\eqsp,
\]
where $\Delta V_i^{\dim}= V(X_{0,i}^{\dim}) -
V(X_{0,i}^{\dim}+\ell\dim^{-1/2}Z^{\dim}_{1,i})$ and where $X_{0}^{\dim}$ is distributed
according to $\pi^{\dim}$ and independent of the standard $d$-dimensional Gaussian random
variable $Z_1^{\dim}$. Following the same steps as in the proof of \Cref{lem:approx_ratio}
yields:
\begin{equation}
\label{eq:lim-accratio}
\lim_{\dim\to\plusinfty}\left|\PP\left[\setAccept_1^\dim \right] -  \PE\left[1 \wedge \exp\left\{\Theta^{\dim}\right\}\right]\right| = 0\eqsp,
\end{equation}
where 
\[
\Theta^{\dim} =  -\ell
\dim^{-1/2}\sum_{i=1}^{\dim}Z_{1,i}^{\dim}\Vdot(X_{0,i}^{\dim}) - \ell^2\sum_{i=2}^d
\Vdot(X_{0,i}^{\dim})^2/(4\dim) +
2(\dim-1)\PE\left[\zeta^{\dim}(X_{0,1}^{\dim},Z_{1,1}^{\dim})\right] \eqsp.
\]
Conditional on $X_{0}^{\dim}$, $\Theta^{\dim}$ is a one dimensional Gaussian random variable with mean
$\mu_{\dim}$ and variance $\sigma^2_{\dim}$, defined by
\begin{align*}
\mu_{\dim}      &= - \ell^2\sum_{i=2}^d \Vdot(X_{0,i}^{\dim})^2/(4\dim) + 2(\dim-1)\PE\left[\zeta^{\dim}(X_{0,1}^{\dim},Z_{1,1}^{\dim})\right]\\
\sigma^2_{\dim} &= \ell^2\dim^{-1}\sum_{i=1}^{\dim}\Vdot(X_{0,i}^{\dim})^2\eqsp.
\end{align*}
Therefore, since for any $G\sim\mathcal{N}(\mu,\sigma^2)$, $\PE[1\wedge \exp(G)] =
\Phi(\mu/\sigma) +  \exp(\mu+\sigma^2/2)\Phi(-\sigma-\mu/\sigma)$, taking the expectation
conditional on $X^d_0$, we have
\begin{align*}
\PE\left[1 \wedge \exp\left\{\Theta^{\dim}\right\}\right] 
&= \PE\left[\Phi(\mu_{\dim}/\sigma_{\dim}) +\exp(\mu_{\dim}+\sigma_{\dim}^2/2)\Phi(-\sigma_{\dim}-\mu_{\dim}/\sigma_{\dim})\right]\\
&= \PE\left[\Gamma(\sigma^2_{\dim},-2\mu_{\dim})\right]\eqsp,
\end{align*}
where the function $\Gamma$ is defined in \eqref{eq:defGamma}. By \Cref{lem:mean:zeta} and the law of large numbers, almost surely, $\lim_{\dim\to\plusinfty} \mu_{\dim} = -\ell^2I/2$ and $\lim_{\dim\to\plusinfty} \sigma^2_{\dim} = \ell^2I$. Thus, as $\Gamma$ is bounded, by Lebesgue's dominated convergence theorem:
\[
\lim_{\dim\to\plusinfty}\PE\left[1 \wedge \exp\left\{\Theta^{\dim}\right\}\right] = 2\Phi\left(-\ell\sqrt{I}/2\right)\eqsp.
\]
The proof is then completed by \eqref{eq:lim-accratio}.
\end{proof}

\subsection{Proof of \Cref{prop:tight}}
\label{sec:proof:tightness}
By Kolmogorov's criterion it is enough to prove that there exists a non-decreasing function $\gamma:\mathbb{R}_+ \to \mathbb{R}_+$ such that for all $\dim\ge 1$ and all $0\le s\le t$,
\[
\PE\left[\left(Y_{t,1}^{\dim}-Y_{s,1}^{\dim}\right)^4\right] \le \gamma(t)(t-s)^2\eqsp.
\]
The inequality is straightforward for all $0\le s\le t$ such that $\lfloor\dim s\rfloor = \lfloor\dim t\rfloor$. For all $0\le s\le t$ such that $\lceil\dim s\rceil \le  \lfloor\dim t\rfloor$,
\begin{equation*}
Y_{t,1}^{\dim}-Y_{s,1}^{\dim}
=X_{\floor{\dim t},1}^{\dim}-X_{\ceil{\dim s},1}^{\dim} + \frac{\dim t - \lfloor\dim t\rfloor}{\sqrt{\dim}} \ell Z^{\dim}_{\lceil \dim t\rceil,1} \1_{\setAccept^{\dim}_{\lceil \dim t\rceil}}  + \frac{\ceil{\dim s} - \dim s}{\sqrt{\dim}} \ell  Z^{\dim}_{\lceil \dim s\rceil,1} \1_{\setAccept^{\dim}_{\lceil \dim s\rceil}}\eqsp.
\end{equation*}
Then by the Hölder inequality,
\begin{align*}
\PE\left[\left(Y_{t,1}^{\dim}-Y_{s,1}^{\dim}\right)^4\right]
&\le C \left((t - s)^2 + \PE\left[\left(X_{\lfloor\dim t\rfloor,1}^{\dim} - X_{\lceil\dim s\rceil,1}^{\dim}\right)^4\right]\right)\eqsp,
\end{align*}
where we have used
\begin{equation*}
\frac{(\dim t - \lfloor\dim t\rfloor)^2}{\dim^2} +  \frac{(\lceil\dim s\rceil - \dim s)^2}{\dim^2}
\leq  \frac{(\dim t - \dim s)^2  + (\ceil{\dim s} - \lfloor\dim t\rfloor)^2}{\dim^2} \leq 2 (t-s)^2 \eqsp.
\end{equation*}
The proof is completed using \Cref{lem:kolmo}.

\begin{lemma}
\label{lem:kolmo}
Assume \Cref{assum:diff:quadratic}. Then, there exists $C>0$ such that, for all $0\le k_1<k_2$,
\[
\PE\left[\left(X_{k_2,1}^{\dim} - X_{k_1,1}^{\dim}\right)^4\right] \le C \sum_{p=2}^4 \frac{(k_2-k_1)^p}{\dim^p} \eqsp.
\]
\end{lemma}

\begin{proof}
For all $0\le k_1< k_2$,
\begin{align*}
\PE\left[\left(X_{k_2,1}^{\dim} - X_{k_1,1}^{\dim}\right)^4\right]
& =  \frac{\ell^4}{\dim^2}\PE\left[\left(\sum_{k=k_1+1}^{k_2}Z_{k,1}^{\dim}- \sum_{k=k_1+1}^{k_2}Z_{k,1}^{\dim}\1_{\left(\setAccept_{k}^{\dim}\right)^c}\right)^4\right]\eqsp.
\end{align*}
Therefore by the Hölder inequality,
\begin{align}
\PE\left[\left(X_{k_2,1}^{\dim} - X_{k_1,1}^{\dim}\right)^4\right]
&\le \frac{24\ell^4}{\dim^2}(k_2-k_1)^2 + \frac{8\ell^4}{\dim^2}\PE\left[\left(\sum_{k=k_1+1}^{k_2}Z_{k,1}^{\dim}\1_{\left(\setAccept_{k}^{\dim}\right)^c}\right)^4\right]\label{eq:deltaX}\eqsp.
\end{align}
The second term can be written:
\[
\PE\left[\left(\sum_{k=k_1+1}^{k_2}Z_{k,1}^{\dim}\1_{\left(\setAccept_{k}^{\dim}\right)^c}\right)^4\right] = \sum \PE\left[\prod_{i=1}^4Z_{m_i,1}^{\dim}\1_{\left(\setAccept_{m_i}^{\dim}\right)^c}\right]\eqsp,
\]
where the sum is over all the quadruplets $(m_i)_{i=1}^4$ satisfying $m_i \in \{k_1+1,\dots,k_2\}$, $i=1,\dots,4$.
The expectation on the right hand side can be upper bounded depending on the cardinality of $\{m_1,\ldots,m_4\}$. For all $1\le j\le 4$, define
\begin{equation}
\label{eq:definitionI-j}
\mathcal{I}_j = \left\{(m_1,\ldots,m_4)\in\left\{k_1+1,\dots,k_2\right\}\;;\; \#\{m_1,\ldots,m_4\}=j\right\} \eqsp.
\end{equation}
Let $(m_1,m_2,m_3,m_4)\in\{k_1+1,\dots,k_2\}^4$  and  $(\tilde{X}_k^{\dim})_{k\ge 0}$ be defined as:
\[
\tilde{X}_0^{\dim} = X_0^{\dim} \quad \mbox{and} \quad \tilde{X}_{k+1}^{\dim} = \tilde{X}_{k}^{\dim} + \1_{k\notin\left\{m_1-1,m_2-1,m_3-1,m_4-1\right\}}\frac{\ell}{\sqrt{\dim}}Z_{k+1}^{\dim}1_{\tilde{\setAccept}_{k+1}^{\dim}}\eqsp,
\]
with $\tilde{\setAccept}_{k+1}^{\dim} = \left\lbrace U_{k+1} \leq \exp \left(\sum_{i=1}^{\dim} \Delta \tilde{V}^{\dim}_{k,i}\right) \right\rbrace$,
where for all $k\ge 0$ and all $1\le i \le \dim$, $\Delta \tilde{V}_{k,i}$ is defined by
\[
\Delta \tilde{V}^{\dim}_{k,i} = V\left(\tilde{X}_{k,i}^{\dim}\right) - V\left(\tilde{X}_{k,i}^{\dim} + \frac{\ell}{\sqrt{\dim}} Z^{\dim}_{k+1,i}\right)\eqsp.
\]
Note that on the event $\bigcap_{j=1}^4 \left\{\setAccept_{m_j}^{\dim}\right\}^c$, the two processes
$(X_k)_{k \geq 0}$ and $(\tilde{X}_k)_{k \geq 0}$ are equal. Let $\mcf$ be the $\sigma$-field generated by $\left(\tilde{X}^d_k\right)_{k\ge 0}$.

\begin{enumerate}[label=(\alph*), wide=0pt, labelindent=\parindent]
\item $\#\{m_1,\ldots,m_4\}=4$, as the $\left\{\left(U_{m_j},Z_{m_j,1}^{\dim}, \dots,Z_{m_j,\dim}^{\dim} \right)\right\}_{1\le j \le 4}$ are independent conditionally to $\mcf$,
\begin{align*}
\PE\left[\prod_{j=1}^4Z_{m_j,1}^{\dim}\1_{\left(\setAccept_{m_j}^{\dim}\right)^c}\middle| \mcf\right] &= \prod_{j=1}^4\PE\left[Z_{m_j,1}^{\dim}\1_{\left(\tilde{\setAccept}_{m_j}^{\dim}\right)^c}\middle| \mcf\right]\eqsp,\\
&= \prod_{j=1}^4\PE\left[Z_{m_j,1}^{\dim}\varphi\left(\sum_{i=1}^{\dim} \Delta \tilde{V}^{\dim}_{m_j-1,i}\right)\middle| \mcf\right]\eqsp.
\end{align*}
where $\varphi(x)= \left(1-\mathrm{e}^x\right)_+$. Since the function $\varphi$ is 1-Lipschitz, we get
\begin{multline*}
\left|\varphi\left(\sum_{i=1}^{\dim} \Delta \tilde{V}^{\dim}_{m_j-1,i}\right) -\varphi\left(-\frac{\ell}{\sqrt{\dim}}\Vdot(\tilde{X}_{m_j-1,1}^{\dim})Z^{\dim}_{m_j,1} + \sum_{i=2}^{\dim} \Delta \tilde{V}^{\dim}_{m_j-1,i}\right)\right|\\
\le\left|\Delta \tilde{V}^{\dim}_{m_j-1,1} + \frac{\ell}{\sqrt{\dim}}\Vdot(\tilde{X}_{m_j-1,1}^{\dim})Z^{\dim}_{m_j,1}\right|\eqsp.
\end{multline*}
Then,
\begin{equation*}
\left|\PE\left[\prod_{j=1}^4Z_{m_j,1}^{\dim}\1_{\left(\setAccept_{m_j}^{\dim}\right)^c}\right]\right|
\leq  \PE\left[\prod_{j=1}^4\left\{A_{m_j}^{\dim} + B_{m_j}^{\dim}\right\}\right]\eqsp,
\end{equation*}
where
\begin{align*}
A_{m_j}^{\dim}&= \PE\left[\left|Z_{m_j,1}^{\dim}\right|\left|\Delta \tilde{V}^{\dim}_{m_j-1,1} + \frac{\ell}{\sqrt{\dim}}\Vdot(\tilde{X}_{m_j-1,1}^{\dim})Z^{\dim}_{m_j,1}\right|\middle|\mcf\right] \eqsp,\\
B_{m_j}^{\dim}&= \left|\PE\left[\!Z_{m_j,1}^{\dim}\left(\!1\!-\!\mathrm{exp}\left\{\!-\frac{\ell}{\sqrt{\dim}}\Vdot(\tilde{X}_{m_j-1,1}^{\dim})Z^{\dim}_{m_j,1} + \sum_{i=2}^{\dim} \Delta \tilde{V}^{\dim}_{m_j-1,i}\right\}\right)_+\middle|\mcf\right]\right|\eqsp.
\end{align*}
By the inequality of arithmetic and geometric means and convex inequalities,
\[
\left|\PE\left[\prod_{j=1}^4Z_{m_j,1}^{\dim}\1_{\left(\setAccept_{m_j}^{\dim}\right)^c}\right]\right|
\le 8\PE\left[\sum_{j=1}^4\left(A_{m_j}^{\dim}\right)^4 + \left(B_{m_j}^{\dim}\right)^4\right]\eqsp.
\]
By \Cref{lem:integrated-DQM}\ref{lem:integrated-DQM-Lp} and the Hölder inequality, there exists $C>0$ such that $\PE\left[\left(A_{m_j}^{\dim}\right)^4\right]\le C\dim^{-2}$. On the other hand, by \cite[Lemma~6]{jourdain:lelievre:miasojedow:2015} since $Z_{m_j,1}^d$ is independent of $\mcf$,
\begin{equation*}
B_{m_j}^{\dim} = \left|\PE\left[\frac{\ell}{\sqrt{\dim}}\Vdot(\tilde{X}_{m_j-1,1}^{\dim})\G\left(\frac{\ell^2}{\dim}\Vdot(\tilde{X}_{m_j-1,1}^{\dim})^2,-2\sum_{i=2}^{\dim} \Delta \tilde{V}^{\dim}_{m_j-1,i}\right)\middle|\mcf\right]\right|\eqsp,
\end{equation*}
where the function $\G$ is defined in \eqref{eq:defG}. By \Cref{assum:diff:quadratic}\ref{assum:X6} and since $\G$ is bounded, $\PE[(B_{m_j}^{\dim})^4]\le C\dim^{-2}$. Therefore $|\PE[\prod_{j=1}^4Z_{m_j,1}^{\dim}\1_{(\setAccept_{m_j}^{\dim})^c}]|\le C\dim^{-2}$, showing that
\begin{equation}
\label{eq:card4}
\sum_{(m_1,m_2,m_3,m_4)\in\mathcal{I}_4}\left|\PE\left[\prod_{i=1}^4Z_{m_i,1}^{\dim}\1_{\left(\setAccept_{m_i}^{\dim}\right)^c}\right]\right|\le \frac{C}{\dim^2}{ k_2-k_1 \choose 4}\eqsp.
\end{equation}
\item $\#\{m_1,\ldots,m_4\}=3$, as the $\left\{\left(U_{m_j},Z_{m_j,1}^{\dim},\ldots,Z_{m_j,\dim}^{\dim}\right)\right\}_{1\le j \le 3}$ are independent conditionally to $\mcf$,
\begin{multline*}
\left|\PE\left[\left(Z_{m_1,1}^{\dim}\right)^2\1_{\left(\setAccept_{m_1}^{\dim}\right)^c}\prod_{j=2}^3Z_{m_j,1}^{\dim}\1_{\left(\setAccept_{m_j}^{\dim}\right)^c}\middle| \mcf\right]\right| \\
\le \PE\left[\left(Z_{m_1,1}^{\dim}\right)^2\middle|
  \mcf\right]\left|\prod_{j=2}^3\PE\left[Z_{m_j,1}^{\dim}\1_{\left(\tilde{\setAccept}_{m_j}^{\dim}\right)^c}\middle| \mcf\right]\right|\le \left|\prod_{j=2}^3\PE\left[Z_{m_j,1}^{\dim}\1_{\left(\tilde{\setAccept}_{m_j}^{\dim}\right)^c}\middle| \mcf\right]\right|\eqsp.
\end{multline*}
Then, following the same steps as above, and using Holder's inequality yields
\begin{align*}
\left|\PE\left[\prod_{j=2}^3Z_{m_j,1}^{\dim}\1_{\left(\setAccept_{m_j}^{\dim}\right)^c}\right]\right| &\le C\PE\left[\sum_{j=2}^3\left(A_{m_j}^{\dim}\right)^2 + \left(B_{m_j}^{\dim}\right)^2\right]\le C\dim^{-1}
\end{align*}
and
\begin{equation}
\label{eq:card3}
\sum_{(m_1,m_2,m_3,m_4)\in\mathcal{I}_3}\left|\PE\left[\prod_{i=1}^4Z_{m_i,1}^{\dim}\1_{\left(\setAccept_{m_i}^{\dim}\right)^c}\right]\right|\le \frac{C}{\dim}{ k_2-k_1 \choose 3}\le\frac{C}{\dim}(k_2-k_1)^3\eqsp.
\end{equation}
\item If $\#\{m_1,\ldots,m_4\}=2$ two cases have to be considered:
\begin{align*}
\PE\left[\left(Z_{m_1,1}^{\dim}\right)^2\1_{\left(\tilde{\setAccept}_{m_1}^{\dim}\right)^c}\left(Z_{m_2,1}^{\dim}\right)^2\1_{\left(\setAccept_{m_2}^{\dim}\right)^c}\right] &\le \PE\left[\left(Z_{m_1,1}^{\dim}\right)^2\right]\PE\left[\left(Z_{m_2,1}^{\dim}\right)^2\right] \le 1\eqsp,\\
\PE\left[\left(Z_{m_1,1}^{\dim}\right)^3\1_{\left(\setAccept_{m_1}^{\dim}\right)^c}Z_{m_2,1}^{\dim}\1_{\left(\setAccept_{m_2}^{\dim}\right)^c}\right] &\le \PE\left[\left|Z_{m_1,1}^{\dim}\right|^3\right]\PE\left[\left|Z_{m_2,1}^{\dim}\right|\right] \le \frac{4}{\pi}\eqsp.
\end{align*}
This yields
\begin{multline}
\label{eq:card2}
\sum_{(m_1,m_2,m_3,m_4)\in\mathcal{I}_2}\left|\PE\left[\prod_{i=1}^4Z_{m_i,1}^{\dim}\1_{\left(\setAccept_{m_i}^{\dim}\right)^c}\right]\right|\\
\le \left(3 + 4\cdot\frac{4}{\pi}\right)(k_2-k_1)(k_2-k_1-1)\le C(k_2-k_1)^2\eqsp.
\end{multline}
\item If $\#\{m_1,\ldots,m_4\}=1$: $\PE\left[\left(Z_{m_i,1}^{\dim}\1_{\left(\setAccept_{m_i}^{\dim}\right)^c}\right)^4\right] \le \PE\left[\left(Z_{m_1,1}^{\dim}\right)^4\right] \le 3$, then
\begin{equation}
\label{eq:card1}
\sum_{(m_1,m_2,m_3,m_4)\in\mathcal{I}_1}\left|\PE\left[\prod_{i=1}^4Z_{m_i,1}^{\dim}\1_{\left(\setAccept_{m_i}^{\dim}\right)^c}\right]\right|\le 3(k_2-k_1)\eqsp.
\end{equation}
\end{enumerate}
The proof is completed by combining \eqref{eq:deltaX} with \eqref{eq:card4}, \eqref{eq:card3}, \eqref{eq:card2} and \eqref{eq:card1}.
\end{proof}

\subsection{Proof of \Cref{propo:reduction_martingale_problem}}
\label{sec:reduction_martingale_problem}
We preface the proof by a preliminary lemma.
\begin{lemma}
\label{lem:law_point_limit}
Assume that \Cref{assum:diff:quadratic} holds. Let $\mu$ be a  limit point of the sequence of laws $(\mu_{\dim})_{\dim\ge 1}$ of $\defEns{(Y_{t,1}^{\dim})_{t\geq 0}, \ \dim \in
  \Nset^*}$. Then for all $t \geq 0$, the pushforward measure of $\mu$ by $W_t$ is $\pdf$.
   \end{lemma}
\begin{proof}
By \eqref{eq:defY},
\[
\lim_{\dim \to \plusinfty} \PE\left[\left|Y^{\dim}_{t,1} -X_{\floor{\dim t},1}^{\dim}\right|\right] = 0\eqsp.
\]
Since $(\mu_{\dim})_{\dim\ge 1}$ converges weakly to $\mu$, for all bounded Lipschitz
function $\psi : \rset \to \rset$, $\PE^{\mu} [ \psi(W_t) ] = \lim_{d\to \plusinfty}  \PE
[\psi(Y^{\dim}_{t,1})] = \lim_{d\to \plusinfty}\PE [\psi( X_{\floor{\dim t},1}^{\dim})]$. The proof is completed upon noting that for all
$\dim \in \Nset^*$ and all $t \geq 0$, $ X_{\floor{\dim t},1}^{\dim}$ is distributed
according to $\pdf$ .
\end{proof}

\begin{proof}[Proof of \Cref{propo:reduction_martingale_problem}]
Let $\mu$ be a limit point of $(\mu_{\dim})_{\dim\ge 1}$. It is straightforward to show that $\mu$ is a solution to the martingale problem associated with $\generator$ if for all $\phi \in C_c^{\infty}(\rset,\rset)$, $m \in \nset^*$, $g : \rset^m \to \rset$ bounded and continuous, and $0 \leq t_1 \leq \dots \leq t_m \leq s \leq t$:
 \begin{equation}
 \label{eq:reduction_martingale_pb1}
  \PE^{\mu}\left[ \left(\phi\left(W_t\right) - \phi\left( W_s \right) - \int_s^t \Lrm \phi \left(W_u \right) \rmd u  \right) g\left(W_{t_1},\dots,W_{t_m} \right) \right] =0 \eqsp.
 \end{equation}
Let $\phi \in C_c^{\infty}(\rset,\rset)$, $m\in \nset^*$, $g:\Rset^m \to \rset$ continuous and bounded, $0 \leq t_1 \leq \dots \leq t_m \leq s \leq t$ and $\Wienerspace_{\Vdot} = \lbrace w \in \Wienerspace | w_u \not \in \setDisconDotV \text{  for  almost every } u \in [s,t]\rbrace$. Note first that $w \in \Wienerspace_{\Vdot}^c$ if and only if $\int_s^t \1_{\setDisconDotV}(w_u ) \rmd u >0 $. 
Therefore, by \Cref{assum:Vdot} and Fubini's theorem:
\begin{equation*}
\PE^{\mu}\left[ \int_s ^t \1_{ \setDisconDotV}(W_u ) \rmd u  \right] = \int_s ^t \PE^{\mu} \left[ \1_{ \setDisconDotV}(W_u ) \right] \rmd u =0\eqsp,
\end{equation*}
showing that $\mu( \Wienerspace_{\Vdot}^c) = 0$. We now prove that on $\Wienerspace_{\Vdot}$,
\begin{equation}
\label{eq:reduction_martingale_problem2}
\Psi_{s,t}:\;w \mapsto  \left\{\phi(w_t) - \phi(w_s) - \int_s^t \Lrm \phi(w_u) \rmd u \right\} g(w_{t_1},\ldots,w_{t_m})
\end{equation}
is continuous. It is clear that it is enough to show that $w \mapsto \int_s^t \Lrm \phi(w_u)
\rmd u $ is continuous on $\Wienerspace_{\Vdot}$. So let $w \in \Wienerspace_{\Vdot}$ and
$\left(w^n\right)_{n\ge 0}$ be a sequence in $\Wienerspace$ which converges to $w$ in the
uniform topology on compact sets. Then by
\Cref{assum:Vdot}, for any $u$ such that $w_u\notin
\setDisconDotV$, $\generator \phi(w_u^n)$ converges to $\Lrm\phi(w_u)$ when $n$ goes to infinity and $\generator \phi$ is bounded. Therefore by
Lebesgue's dominated convergence theorem, $\int_s ^t \Lrm\phi(w_u^n) \rmd u$ converges to
$\int_s ^t \generator \phi(w_u) \rmd u$. Hence, the map defined by \eqref{eq:reduction_martingale_problem2} is continuous on $\Wienerspace_{\Vdot}$. Since $(\mu_{\dim})_{\dim\ge 1}$ converges weakly to $\mu$, by \eqref{eq:reduction_martingale_problem}:
\[
\mu\left(\Psi_{s,t}\right) = \lim_{d \to \plusinfty}\mu^{\dim}\left(\Psi_{s,t}\right)=0\eqsp,
\]
which is precisely \eqref{eq:reduction_martingale_pb1}.
\end{proof}

\subsection{Proof of \Cref{theo:diffusion_limit_RMW}}
\label{sec:proof:weaklimit}
By \Cref{propo:reduction_martingale_problem}, it is enough to check \eqref{eq:reduction_martingale_problem} to prove that $\mu$ is a solution to the martingale problem. The core of the proof of \Cref{theo:diffusion_limit_RMW} is \Cref{propo:decomposition_martingale}, for which  we need two technical lemmata.
\begin{lemma}
\label{lem:fun_nearly_lip}
Let $\Xsf,\Ysf$ and $\Usf$ be $\rset$-valued random variables and $\epsilon>0$. Assume that $\Usf$ is nonnegative and bounded by $1$. Let $g: \rset \to \rset$ be a bounded function on $\rset$ such that for all $(x,y) \in \ocint{-\infty , -\epsilon}^2 \cup \coint{\epsilon,+\infty}^2$, $\abs{g(x)-g(y)}\leq C_g\abs{x-y}$.
\begin{enumerate}[label=(\roman*)]
\item
\label{lem:item:fun_nearly1}
For all $a>0$,
\begin{multline*}
\expe{\Usf\abs{g(\Xsf) - g(Y)}} \leq C_g \expe{\Usf\abs{\Xsf-Y}} \\ +\osc(g) \defEns{\proba{\abs{\Xsf}\leq \epsilon} + a^{-1}\expe{\Usf \abs{\Xsf-Y}} + \proba{\epsilon < \abs{\Xsf} < \epsilon +a}} \eqsp,
\end{multline*}
where $\osc(g) = \sup(g) - \inf(g)$.
\item
\label{lem:item:fun_nearly2}
If there exist $\mu \in \rset$ and  $\sigma,C_{\Xsf} \in \rset_+$ such that
\[
\sup_{x \in \rset} \abs{\proba{\Xsf \leq x} - \Phi((x-\mu)/\sigma) } \leq C_{\Xsf} \eqsp,
\]
then
\begin{multline*}
\expe{\Usf \abs{g(\Xsf) - g(Y)}} \leq C_g \expe{\Usf\abs{\Xsf-Y}} \\
+ 2\osc(g)\defEns{C_{\Xsf} + \sqrt{2\expe{\Usf\abs{\Xsf-Y}}(2 \pi \sigma^2)^{-1/2}} +\epsilon (2 \pi \sigma^2)^{-1/2}} \eqsp.
\end{multline*}
\end{enumerate}
\end{lemma}

\begin{proof}
\begin{enumerate}[label=(\roman*),wide=0pt, labelindent=\parindent]
\item
Consider the following decomposition
\begin{multline*}
\expe{\Usf\abs{g(\Xsf) - g(\Ysf)}} = \expe{\Usf \abs{\parenthese{g(\Xsf)-g(\Ysf)} }\1_{\defEns{(\Xsf,\Ysf) \in \ocint{-\infty,-\epsilon}^2}\cup \defEns{(\Xsf,\Ysf) \in \coint{\epsilon,\plusinfty}^2}}} \\
+
\expe{\Usf
\abs{g(\Xsf)-g(\Ysf)} \parenthese{\1_{\left\{\Xsf \in \ccint{-\epsilon,\epsilon}\right\}} + \1_{
\parenthese{\defEns{\Xsf < -\epsilon}\cap\defEns{\Ysf \geq -\epsilon}} \cup \parenthese{\defEns{\Xsf> \epsilon}\cap\defEns{\Ysf \leq \epsilon}}
}}
}\eqsp.
\end{multline*}
In addition, for all $a >0$,
\begin{multline*}
\parenthese{\defEns{\Xsf < -\epsilon}\cap\defEns{\Ysf \geq -\epsilon}} \cup \parenthese{\defEns{\Xsf> \epsilon}\cap\defEns{\Ysf \leq \epsilon}} \\ \subset \defEns{\epsilon < \abs{\Xsf} < \epsilon +a}
\cup \parenthese{\defEns{\abs{\Xsf} \geq \epsilon+a} \cap \defEns{\abs{\Xsf - \Ysf} \geq a}}\eqsp.
\end{multline*}
Then using that $\Usf \in \coint{0,1}$, we get
\[
\expe{\Usf \abs{g(\Xsf) - g(\Ysf)}} \leq C_g\expe{\Usf \abs{\Xsf-\Ysf}}  +\osc(g)\parenthese{\proba{\abs{\Xsf} < \epsilon +a } + a^{-1} \expe{\Usf \abs{\Xsf-\Ysf}}} \eqsp.
\]
\item  The result is straightforward if $\expe{\Usf \abs{\Xsf-\Ysf}} =0$. Assume $\expe{\Usf \abs{\Xsf-\Ysf}} >0$. Combining the additional assumption and the previous result,
\begin{multline*}
\expe{\Usf\abs{g(\Xsf) - g(\Ysf)}} \leq C_g\expe{\Usf\abs{\Xsf-\Ysf}} \\ 
+\osc(g)\left\{2 C_\Xsf + 2 (\epsilon +a)(2\pi \sigma^2)^{-1/2}+ a^{-1} \expe{\Usf\abs{\Xsf-\Ysf}}\right\}\eqsp.
\end{multline*}
As this result holds for all $a>0$, the proof is concluded by setting $a = \sqrt{\expe{\Usf\abs{\Xsf-\Ysf}}(2\pi \sigma^2)^{1/2}/2}$.
\end{enumerate}
\end{proof}
\begin{lemma}
\label{lem:approx_first_term_decomposition_martingale}
Assume \Cref{assum:diff:quadratic} holds. Let
$X^d$ be distributed according to $\target^{d}$ and $Z^d$ be a
$d$-dimensional standard Gaussian random variable, independent of $X^d$. Then, $\lim_{d
  \to \plusinfty} \Erm^{\dim} = 0$, where
\begin{equation*}
\Erm^{\dim}=  \expe{\abs{\Vdot(X_1^d) \defEns{\mathcal{G}\parenthese{\frac{\ell^2}{\dim}\Vdot(X_1^d)^2, 2 \sum_{i=2}^d \Delta V_i^d}-\mathcal{G}\parenthese{\frac{\ell^2}{\dim}\Vdot(X_1^d)^2, 2 \sum_{i=2}^d b_i^d}}} }  \eqsp,
\end{equation*}
$\Delta V_i^d  $ and $b_i^d$ are resp. given by \eqref{eq:Delta_V} and \eqref{eq:bdi}.
\end{lemma}

\begin{proof}
Set for all $d \geq 1$, $\barY_d = \sum_{i=2}^d \Delta V_i^d$ and $\barX_d = \sum_{i=2}^d b_i^d$. By \eqref{eq:defG}, $\partial_b \mathcal{G} (a,b) = -\mathcal{G}(a,b)/2 + \exp(-b^2/8a)/(2\sqrt{2 \pi a})$. As $\mathcal{G}$ is bounded and $x \mapsto x\exp(-x)$ is bounded on $\rset_+$, we get
$\sup_{a \in \rset_+ ; \abs{b} \geq a^{1/4}} \partial_b \mathcal{G}(a,b) < \plusinfty$.
Therefore, there exists $C\geq 0$ such that, for all $a \in \rset_+$and $(b_1,b_2) \in \ooint{- \infty , -a^{1/4}}^2 \cup \ooint{a^{1/4}, \plusinfty}^2$,
\begin{equation}
\label{eq:G_lip_in_b}
\abs{\mathcal{G}(a,b_1)-\mathcal{G}(a,b_2)} \leq C \abs{b_1-b_2}\eqsp.
\end{equation}
By definition of $b_i^d$ \eqref{eq:bdi},  $\barX_d$ may be expressed as $\barX_d = \sigma_d \barS_d +\mu_d$, where
\begin{align*}
\mu_d &= 2 (d-1) \expe{\zeta^{\dim}(X_1^{\dim},Z_1^{\dim})} - \frac{\ell^2(d-1)}{4d} \expe{\Vdot(X_1^{\dim})^2}\eqsp, \\
\sigma^2_d &= \ell^2 \expe{\Vdot(X_1^{\dim})^2} + \frac{\ell^4}{16d} \expe{\parenthese{\Vdot(X_1^{\dim})^2 - \expe{\Vdot(X_1^{\dim})^2}}^2}\eqsp,  \\
\barS_d &= (\sqrt{d} \sigma_d)^{-1}\sum_{i=2}^d \beta_i^d\eqsp, \\
\beta_i^d &= -\ell Z_i^d \Vdot(X_i^d) - \frac{\ell^2}{4\sqrt{\dim}}\parenthese{\Vdot(X_i^d)^2 - \expe{\Vdot(X_i^d)^2}} \eqsp.
\end{align*}
By \Cref{assum:diff:quadratic}\ref{assum:X6} the Berry-Essen Theorem \cite[Theorem 5.7]{petrov:1995} can be applied to $\barS_d$. Then, there exists a universal constant $C$ such that for all $\dim>0$,
\[
\sup_{x \in \rset} \abs{\proba{\parenthese{\frac{d}{d-1}}^{1/2}\barS_d \leq x}-\Phi(x)} \leq C/ \sqrt{d} \eqsp.
\]
It follows that
\[
\sup_{x \in \rset} \abs{\proba{\barX_d \leq x}-\Phi((x-\mu_d)/\tilde{\sigma}_d)} \leq C/ \sqrt{d} \eqsp,
\]
where $\tilde{\sigma}^2_d = (d-1)\sigma_d^2/d$.
By this result and \eqref{eq:G_lip_in_b}, \Cref{lem:fun_nearly_lip} can be applied to obtain a constant $C \geq 0$, independent of $\dim$, such that:
\begin{multline*}
\expe{\abs{\mathcal{G}\parenthese{\ell^2\Vdot(X_1^d)^2/d, 2 \barY_d}-\mathcal{G}\parenthese{\ell^2\Vdot(X_1^d)^2/d, 2 \barX_d}} \sachant{X_1^d}} \\
\leq  C \parenthese{ \varepsilon_{\dim} + d^{-1/2} + \sqrt{2\varepsilon_{\dim}(2 \pi \tilde{\sigma}_d^2)^{-1/2}} +  \sqrt{\ell|\Vdot(X_1^d)|/( 2 \pi d^{1/2} \tilde{\sigma}_d^2 )} } \eqsp,
\end{multline*}
where $\varepsilon_{\dim} = \expe{\abs{\barX_d-\barY_d}}$. Using this result, we have
\begin{multline}
\label{eq:approx_first_term_decomp_martingale_3}
\Erm^{\dim}\leq  C \left\{\parenthese{ \varepsilon_{\dim} + d^{-1/2} + \sqrt{2\varepsilon_{\dim}(2 \pi \tilde{\sigma}^2_d)^{-1/2}} } \expe{|\Vdot(X_1^d)|}\right.   \\ \left.+  \ell^{1/2}\expe{|\Vdot(X_1^d)|^{3/2}} ( 2 \pi d^{1/2} \tilde{\sigma}^2_d )^{-1/2} \right\} \eqsp.
\end{multline}
By \Cref{lem:interm_prop_approx_ratio}, $\varepsilon_{\dim}$ goes to $0$ as $d$ goes to infinity, and by
\Cref{assum:diff:quadratic}\ref{assum:X6} $\lim_{d \to \plusinfty} \sigma^2_d =
\ell^2\expe{\Vdot(X)^2} $. Combining these results with
\eqref{eq:approx_first_term_decomp_martingale_3}, it follows that $\Erm^{\dim}$ goes to $0$
when $d$ goes to infinity.
\end{proof}
For all $n\ge 0$, define $\mcf_{n}^{\dim} = \sigma( \lbrace X_k^d,k\leq n \rbrace) $
and for all $\phi \in C_c^{\infty}(\rset,\rset)$,
\begin{multline}
\label{eq:def:martingale}
M^{\dim}_n(\phi) = \frac{\ell}{\sqrt{\dim}} \sum_{k=0}^{n-1} \phi'(X_{k,1}^d)\left\{Z_{k+1,1}^d \1_{\setAccept^d_{k+1}} - \CPE{Z_{k+1,1}^d \1_{\setAccept^d_{k+1}}}{\mcf_k^d} \right\} \\
+  \frac{\ell^2}{2 \dim} \sum_{k=0}^{n-1} \phi''(X_{k,1}^d)\left\{(Z_{k+1,1}^d)^2
\1_{\setAccept^d_{k+1}} - \CPE{(Z_{k+1,1}^d)^2 \1_{\setAccept^d_{k+1}}}{\mcf_k^d} \right\}  \eqsp.
\end{multline}

\begin{proposition}
\label{propo:decomposition_martingale}
Assume \Cref{assum:diff:quadratic} and \Cref{assum:Vdot} hold. Then, for all $s \leq t$ and all $\phi \in C_c^{\infty}(\rset,\rset)$,
\[
\lim_{\dim \to \plusinfty} \PE \left[ \abs{\phi(Y_{t,1}^{\dim} ) -\phi(Y_{s,1}^{\dim} ) - \int_s ^t \Lrm \phi(Y_{r,1}^{\dim}) \rmd r -\left(M^{\dim}_{\ceil{\dim t}}(\phi) - M^{\dim}_{\ceil{\dim s}}(\phi)\right)} \right] = 0\eqsp.
\]
\end{proposition}

\begin{proof}
First, since $\rmd Y^{d}_{r,1} = \ell\sqrt{d}Z^{d}_{\ceil{dr},1} \1_{\setAccept_{\ceil{dr}}^{d}}\rmd r$,
\begin{equation}
\label{eq:dev_phiY}
\phi(Y_{t,1}^{d}) -  \phi(Y_{s,1}^{d}) =  \ell\sqrt{\dim}\int_s ^t \phi'(Y^{d}_{r,1})Z^{d}_{\ceil{dr},1} \1_{\setAccept_{\ceil{dr}}^{d}}\rmd r \eqsp.
\end{equation}
As $\phi$ is $C^3$, using \eqref{eq:defY} and a Taylor expansion, for all $r \in \ccint{s,t}$ there exists $\chi_r \in \ccint{X^{d}_{\floor{dr},1},Y^{d}_{r,1}}$ such that:
\begin{multline*}
\phi'(Y^{d}_{r,1}) = \phi'(X^{d}_{\floor{dr},1})+\frac{\ell}{\sqrt{\dim}}(dr -\floor{dr})\phi''(X^{d}_{\floor{dr},1}) Z^{d}_{\ceil{dr},1} \1_{\setAccept_{\ceil{dr}}^{d}} \\ +  \frac{\ell^2}{2\dim}(dr -\floor{dr})^2\phi^{(3)}(\chi_r) \left(Z^{d}_{\ceil{dr},1}\right)^2 \1_{\setAccept_{\ceil{dr}}^{d}} \eqsp.
\end{multline*}
Plugging this expression into \eqref{eq:dev_phiY} yields:
\begin{align*}
\phi(Y_{t,1}^{d}) -  \phi(Y_{s,1}^{d})& =  \ell\sqrt{\dim}\int_s ^t \phi'(X^{d}_{\floor{dr},1})Z^{d}_{\ceil{dr},1} \1_{\setAccept_{\ceil{dr}}^{d}}\rmd r \\
&+ \ell^2\int_s ^t (dr -\floor{dr})\phi''(X^{d}_{\floor{dr},1})(Z^{d}_{\ceil{dr},1})^2 \1_{\setAccept_{\ceil{dr}}^{d}}\rmd r\\
&+ \frac{\ell^3}{2\sqrt{\dim}}\int_s ^t (dr -\floor{dr})^2\phi^{(3)}(\chi_r)(Z^{d}_{\ceil{dr},1})^3 \1_{\setAccept_{\ceil{dr}}^{d}}\rmd r\eqsp.
\end{align*}
As $\phi^{(3)}$ is bounded,
\[
\lim_{\dim \to \plusinfty} \PE \left[ \abs{\dim^{-1/2}\int_s ^t (dr -\floor{dr})^2\phi^{(3)}(\chi_r)(Z^{d}_{\ceil{dr},1})^3 \1_{\setAccept_{\ceil{dr}}^{d}}\rmd r}\right] = 0\eqsp.
\]
On the other hand, $I = \int_s^t \phi''(X^{d}_{\floor{dr},1})(dr -\floor{dr}) (Z^{d}_{\ceil{dr},1})^2 \1_{\setAccept_{\ceil{dr}}^{d}} \rmd r  = I_1 +I_2$ with
\begin{align*}
I_1&= \int_{s} ^{\ceil{ds}/d} + \int_{\floor{dt}/d} ^{t} \phi''(X^{d}_{\floor{dr},1})(dr -\floor{dr} -1/2)(Z^{d}_{\ceil{dr},1})^2 \1_{\setAccept_{\ceil{dr}}^{d}} \rmd r \\
I_2&= \frac{1}{2} \int_{s}^{t} \phi''(X^{d}_{\floor{dr},1}) (Z^{d}_{\ceil{dr},1})^2 \1_{\setAccept_{\ceil{dr}}^{d}} \rmd r \eqsp.
\end{align*}
Note that
\begin{multline*}
I_1 = \frac{1}{2\dim}(\ceil{ds}-ds)(ds-\floor{ds})\phi''(X^{d}_{\floor{ds},1}) (Z^{d}_{\ceil{ds},1})^2 \1_{\setAccept_{\ceil{ds}}^{d}} \\
 +\frac{1}{2\dim}(\ceil{dt}-dt)(dt-\floor{dt})\phi''(X^{d}_{\floor{dt},1}) (Z^{d}_{\ceil{dt},1})^2 \1_{\setAccept_{\ceil{dt}}^{d}}
\end{multline*}
showing, as $\phi''$ is bounded, that $\lim_{\dim \to \plusinfty} \PE [|I_1|]= 0$.
Therefore, 
\[
\lim_{d \to \plusinfty} \mathbb{E}\left[\left| \phi(Y_{t,1}^{\dim} ) - \phi(Y_{s,1}^{\dim} ) - I_{s,t} \right| \right] = 0 \eqsp,
\]
where 
\[
I_{s,t} =  \int_s ^t \left\{\ell \sqrt{\dim} \phi'(X^{\dim}_{\floor{\dim r },1}) Z^{\dim}_{\ceil{\dim r},1} + \ell^2 \phi''(X^{\dim}_{\floor{\dim r },1}) (Z^{\dim}_{\ceil{\dim r},1})^2/2\right\}\1_{\setAccept_{\ceil{\dim r}}^{\dim}} \rmd r\eqsp.
\]
Write
\[
I_{s,t} - \int_s ^t \Lrm \phi(Y_{r,1}^{\dim}) \rmd r - \left(M^{\dim}_{\ceil{\dim t}}(\phi) - M^{\dim}_{\ceil{\dim s}}(\phi)\right)
= T^{\dim}_1 + T_2^{\dim} + T_3^{\dim} - T_4^{\dim} + T_5^{\dim}\eqsp,
\]
where
\begin{align*}
T^{\dim}_1 &=  \int_s ^t  \phi'(X^{\dim}_{\floor{\dim r },1}) \left( \ell \sqrt{\dim}  \
\CPE{Z_{\ceil{\dim r },1}^d \1_{\setAccept^d_{\ceil{\dim r }}}}{ \mcf_{\floor{\dim r}}^\dim}  + \frac{h(\ell)}{2}\Vdot(X^{\dim}_{\floor{\dim r },1})\right) \rmd r\eqsp,\\
T^{\dim}_2 & = \int_s ^t  \phi''(X^{\dim}_{\floor{\dim r },1}) \left( \frac{\ell^2}{2} \  \PE \left[ (Z_{\ceil{\dim r },1}^d)^2 \1_{\setAccept^d_{\ceil{\dim r }}}  \sachant{ \mcf_{\floor{\dim r}} ^d} \right] -  \frac{h(\ell)}{2} \right) \rmd r \eqsp,\\
T^{\dim}_3 & = \int_s ^t \left(\Lrm \phi(Y^{\dim}_{\floor{\dim r}/\dim,1}) - \Lrm \phi(Y^{\dim}_{r,1})\right) \rmd r\eqsp, \\
T^{\dim}_4& = \frac{\ell(\ceil{\dim t} - \dim t)}{\sqrt{\dim}} \phi'(X^{\dim}_{\floor{\dim t },1}) \left(  Z_{\ceil{\dim t },1}^d \1_{\setAccept^d_{\ceil{\dim t }}} -  \PE \left[ Z_{\ceil{\dim t },1}^d \1_{\setAccept^d_{\ceil{\dim t }}}  \sachant{ \mcf_{\floor{\dim t}} ^d} \right] \right) \\
& \;\;+ \frac{\ell^2 (\ceil{\dim t} - \dim t)}{2 \dim } \phi''(X^{\dim}_{\floor{\dim t },1}) \left(  (Z_{\ceil{\dim t },1}^d)^2 \1_{\setAccept^d_{\ceil{\dim t }}} \!\!-  \PE \left[ (Z_{\ceil{\dim t },1}^d)^2 \1_{\setAccept^d_{\ceil{\dim t }}}  \sachant{ \mcf_{\floor{\dim t}} ^d} \right] \right)\eqsp, \\
T^{\dim}_5& = \frac{\ell(\ceil{\dim s} - \dim s)}{\sqrt{\dim}} \phi'(X^{\dim}_{\floor{\dim s },1}) \left(  Z_{\ceil{\dim s },1}^d \1_{\setAccept^d_{\ceil{\dim s }}} -  \PE \left[ Z_{\ceil{\dim s },1}^d \1_{\setAccept^d_{\ceil{\dim s }}}  \sachant{ \mcf_{\floor{\dim s}} ^d} \right] \right) \\
&  \;\;+ \frac{\ell^2 (\ceil{\dim s} - \dim s)}{2 \dim } \phi''(X^{\dim}_{\floor{\dim s },1}) \left(  (Z_{\ceil{\dim s },1}^d)^2 \1_{\setAccept^d_{\ceil{\dim s }}} \!\!-  \PE \left[ (Z_{\ceil{\dim s },1}^d)^2 \1_{\setAccept^d_{\ceil{\dim s }}}  \sachant{ \mcf_{\floor{\dim s}} ^d} \right] \right) \eqsp.
\end{align*}
It is now proved that for all $1\le i\le 5$, $\lim_{d \to \plusinfty} \PE[ | T^d_i|] =0$.
First, as $\phi'$ and $\phi''$ are bounded,
\begin{equation}
\label{eq:T4_T5}
\PE\left[\abs{T^{\dim}_4 }+\abs{ T^{\dim}_5} \right] \leq C \dim^{-1/2} \eqsp.
\end{equation}
Denote for all $r \in \ccint{s,t}$ and $d \geq 1$,
\begin{align*}
\Delta V_{r,i}^d &= V\left(X_{\floor{dr},i}^{\dim}\right) - V\left(X_{\floor{dr},i}^{\dim} +
  \ell\dim^{-1/2} Z^{\dim}_{\ceil{dr},i}\right) \\
 \Xi_r^d &= 1 \wedge
      \exp\left\{-\ell Z_{\ceil{\dim r },1}^{d}\Vdot(X_{ \floor{\dim r}
          ,1}^{d})/\sqrt{\dim} +\sum_{i=2}^d  \varb_{\floor{\dim r} , i}^{d}\right\} \eqsp,\\
\Upsilon_r^d &= 1 \wedge
      \exp\left\{-\ell Z_{\ceil{\dim r },1}^{d}\Vdot(X_{ \floor{\dim r}
          ,1}^{d})/\sqrt{\dim} +\sum_{i=2}^d \Delta V_{r,i}^d\right\} \eqsp,
\end{align*}
where for all $k,i\ge 0$, $b_{k,i}^d=b^d(X_{k,i}^d,Z_{k+1,i}^d)$, and for all $x,z \in
\rset$, $b^d(x,y)$ is given by \eqref{eq:bdi}.
 By the triangle inequality,
\begin{equation}
\label{eq:treatment_T1}
\abs{T_1^d}  \leq \int_s ^t  \left|\phi'(X^{\dim}_{\floor{\dim r },1})\right|(A_{1,r} + A_{2,r} + A_{3,r}) \rmd r \eqsp,
\end{equation}
where
\begin{align*}
  A_{1,r} &= \abs{\ell \sqrt{\dim} \  \PE \left[
    Z_{\ceil{\dim r },1}^d \left( \1_{\setAccept^d_{\ceil{\dim r }}} - \Upsilon_r^d
    \right)  \sachant{ \mcf_{\floor{\dim r}} ^d} \right]}\eqsp, \\
  A_{2,r} &= \abs{\ell \sqrt{\dim} \  \PE \left[
    Z_{\ceil{\dim r },1}^d \left(  \Upsilon_r^d
      -\Xi_r^d
    \right)  \sachant{ \mcf_{\floor{\dim r}} ^d} \right]}\eqsp, \\
A_{3,r}& =  \abs{\ell \sqrt{\dim}  \ \PE \left[ Z_{\ceil{\dim r },1}^d \Xi_r^{\dim} \sachant{ \mcf_{\floor{\dim r}} ^d} \right] + \Vdot(X^{\dim}_{\floor{\dim r },1}) h(\ell)/2 } \eqsp.
  \end{align*}
Since $t \mapsto 1 \wedge \exp(t)$ is $1$-Lipschitz, by
\Cref{lem:integrated-DQM}\ref{lem:integrated-DQM-Lp} $\PE[\abs{A_{1,r}^{\dim}}]$ goes to
$0$ as $d\to \plusinfty$
for almost all $r$. So by
the Fubini theorem, the
first term  in \eqref{eq:treatment_T1} goes to $0$ as $d \to \plusinfty$.
For $A_{2,r}^{\dim}$, by \cite[Lemma 6]{jourdain:lelievre:miasojedow:2015},
\begin{multline*}
\expe{\abs{A_{2,r}^{\dim}}} \leq \mathbb{E}\left[\left| \ell^2\Vdot(X_{\floor{dr},1}^d)\left\{\mathcal{G}\parenthese{\frac{\ell^2\Vdot(X_{\floor{dr},1}^d)^2}{d}, 2\sum_{i=2}^d \Delta V_{r,i}^d }\right.\right.\right. \\
\left.\left.\left.-\mathcal{G}\parenthese{\frac{\ell^2\Vdot(X_{\floor{dr},1}^d)^2}{d}, 2 \sum_{i=2}^d  \varb_{\floor{\dim r} , i}^{d}} \right\}\right|\right]\eqsp,
\end{multline*}
where $\mathcal{G}$ is defined in \eqref{eq:defG}. By  \Cref{lem:approx_first_term_decomposition_martingale}, this expectation  goes to zero
when $d$ goes to infinity. Then by the Fubini theorem and the Lebesgue dominated
convergence theorem, the second term of \eqref{eq:treatment_T1} goes $0$ as $d \to
\plusinfty$. For the last term, by \cite[Lemma 6]{jourdain:lelievre:miasojedow:2015} again:
\begin{multline}
\ell \sqrt{d} \ \PE \left[ Z_{\ceil{\dim r },1}^d \Xi_r^{\dim} \sachant{
    \mcf_{\floor{\dim r}} ^d} \right]  = -\ell^2\Vdot(X_{ \floor{\dim r}  ,1}^{d}) \\
\label{eq:treatment_T1_1}
 \times \mathcal{G}\left(
  \frac{\ell^2}{\dim}\sum_{i=1}^{d} \Vdot(X_{\floor{\dim r},i}^{d})^2,\frac{\ell^2}{2\dim}\sum_{i=2}^{d} \Vdot(X_{\floor{\dim r},i}^{d})^2 -4(d-1) \PE\left[\zeta^{\dim}(X,Z)\right]\right) \eqsp,
\end{multline}
where  $ X$ is distributed according to $\pdf$ and $Z$ is a standard Gaussian random variable independent of $X$. As $\G$ is continuous on $\rset_+ \times \rset \setminus \defEns{0,0}$ (see \cite[Lemma 2]{jourdain:lelievre:miasojedow:2015}), by \Cref{assum:diff:quadratic}\ref{assum:X6}, \Cref{lem:mean:zeta} and the law of large numbers, almost surely,
\begin{multline}
\label{eq:limG}
\lim_{d \to \plusinfty} \ell^2\mathcal{G}\left(\frac{\ell^2}{\dim}\sum_{i=1}^{d} \Vdot(X_{\floor{\dim r},i}^{d})^2, \frac{\ell^2}{2\dim}\sum_{i=2}^{d} \Vdot(X_{\floor{\dim r},i}^{d})^2 -4(d-1) \PE\left[\zeta^{\dim}(X,Z)\right] \right) \\ = \ell^2\mathcal{G}\left(\ell^2\PE_{}[\Vdot(X)^2],\ell^2\PE_{}[\Vdot(X)^2]\right) = h(\ell)/2\eqsp,
\end{multline}
where $h(\ell)$ is defined in \eqref{eq:defhK}. Therefore by Fubini's Theorem,
\eqref{eq:treatment_T1_1} and Lebesgue's dominated convergence theorem, the last term of
\eqref{eq:treatment_T1} goes to $0$ as $d$ goes to infinity. The proof for $T_2^d$ follows the same lines. By the triangle inequality,
\begin{multline}
\label{eq:treatment_T2}
\abs{T_2^{d}}  \leq \abs{\int_s ^t  \phi''(X^{\dim}_{\floor{\dim r },1})(\ell^2/2) \  \PE \left[ (Z_{\ceil{\dim r },1}^d )^2 \left( \1_{\setAccept^d_{\ceil{\dim r }}} -\Xi_r^{\dim} \right) \sachant{ \mcf_{\floor{\dim r}} ^d} \right]  \rmd r}\\
+ \abs{\int_s ^t  \phi''(X^{\dim}_{\floor{\dim r },1}) \left( (\ell^2/2) \ \PE \left[
      (Z_{\ceil{\dim r },1}^d)^2 \Xi_r^{\dim} \sachant{ \mcf_{\floor{\dim r}} ^d} \right]
    -  h(\ell)/2 \right) \rmd r} \eqsp.
\end{multline}
 By Fubini's Theorem,  Lebesgue's dominated convergence theorem and \Cref{lem:approx_ratio}, the expectation of the first term goes to zero when $d$ goes to infinity. For the second term,  by \cite[Lemma 6 (A.5)]{jourdain:lelievre:miasojedow:2015},
\begin{multline}
\label{eq:treatment_T2_1}
( \ell^2/2) \PE \left[ (Z_{\ceil{\dim r },1}^d)^2 1 \wedge \exp\left\{ -\frac{\ell Z_{\ceil{\dim r },1}^{d}}{\sqrt{\dim}}\Vdot(X_{ \floor{\dim r}  ,1}^{d}) +\sum_{i=2}^d  \varb_{\floor{\dim r} , i}^{d} \right\} \sachant{ \mcf_{\floor{\dim r}} ^d} \right] \\
= (B_1 + B_2 - B_3)/2 \eqsp,
\end{multline}
where
\begin{align*}
B_1 & = \ell^2\Gamma \left(  \frac{\ell^2}{\dim}\sum_{i=1}^{d} \Vdot(X_{\floor{\dim r},i}^{d})^2, \frac{\ell^2}{2\dim}\sum_{i=2}^{d} \Vdot(X_{\floor{\dim r},i}^{d})^2-4(d-1) \PE_{}\left[\zeta^{\dim}(X,Z)\right]\right)\eqsp,\\
B_2 & = \frac{\ell^4 \Vdot(X_{ \floor{\dim r}  ,1}^{d})^2}{d}\mathcal{G}\left(
  \frac{\ell^2}{\dim}\sum_{i=1}^{d} \Vdot(X_{\floor{\dim r},i}^{d})^2, \frac{\ell^2}{2\dim}\sum_{i=2}^{d} \Vdot(X_{\floor{\dim r},i}^{d})^2 -4(d-1) \PE_{}\left[\zeta^{\dim}(X,Z)\right]
  \right)\eqsp,\\
B_3 & = \frac{\ell^4 \Vdot(X_{ \floor{\dim r}  ,1}^{d})^2}{d} \left(2 \pi \ell^2\sum_{i=1}^{d} \Vdot(X_{\floor{\dim r},i}^{d})^2 /d\right)^{-1/2}\\
&\qquad\times \exp\left\{- \frac{\left[-(d-1)\PE_{}[2\zeta^{\dim}(X,Z)] + (\ell^2/(4d)) \sum_{i=2}^d \Vdot(X_{ \floor{\dim r}  ,i}^{d})^2\right]^2}{2\ell^2\sum_{i=1}^{d}\Vdot(X_{\floor{\dim r},i}^{d})^2 /d }\right\}\eqsp,\\
\end{align*}
where $\Gamma$ is defined in \eqref{eq:defGamma}. As $\Gamma$ is continuous on $\rset_+ \times \rset \setminus \defEns{0,0}$ (see \cite[Lemma 2]{jourdain:lelievre:miasojedow:2015}), by \Cref{assum:diff:quadratic}\ref{assum:X6}, \Cref{lem:mean:zeta} and the law of large numbers, almost surely,
\begin{multline}
\label{eq:first_treat_gamma_T2}
\lim_{d \to \plusinfty}
\ell^2 \Gamma\left(
  \frac{\ell^2}{\dim}\sum_{i=1}^{d} \Vdot(X_{\floor{\dim r},i}^{d})^2,\frac{\ell^2}{2\dim} \sum_{i=2}^{d} \Vdot(X_{\floor{\dim r},i}^{d})^2 -4(d-1) \PE_{}\left[\zeta^{\dim}(X,Z)\right]\right) \\
  = \ell^2\Gamma\left(\ell^2\PE_{}[\Vdot(X)^2],\ell^2\PE_{}[\Vdot(X)^2]\right) = h(\ell) \eqsp.
  \end{multline}
By \Cref{lem:mean:zeta}, by \Cref{assum:diff:quadratic}\ref{assum:X6} and the law of large numbers, almost surely,
  \begin{multline*}
  \lim_{d\to\plusinfty}\exp\left\{- \frac{\left[-(d-1)\PE_{}[2\zeta^{\dim}(X,Z)] + (\ell^2/(4d)) \sum_{i=2}^d \Vdot(X_{ \floor{\dim r}  ,i}^{d})^2\right]^2}{2\ell^2\sum_{i=1}^{d}\Vdot(X_{\floor{\dim r},i}^{d})^2 /d }\right\} \\
  = \exp\left\{-\frac{\ell^2}{8}\PE_{}[\Vdot(X)^2]\right\}\eqsp.
  \end{multline*}
  Then, as $\mathcal{G}$ is bounded on $\rset_+ \times \rset$,
  \begin{equation}
  \label{eq:second_treat_gamma_T2}
  \lim_{d\to\plusinfty} \PE\left[\left|\int_s ^t  \phi''(X^{\dim}_{\floor{\dim r },1}) \left( B_2-B_3 \right) \rmd r\right|\right] = 0\eqsp.
  \end{equation}
   Therefore, by Fubini's Theorem,
   \eqref{eq:treatment_T2_1}, \eqref{eq:first_treat_gamma_T2}, \eqref{eq:second_treat_gamma_T2}
   and Lebesgue's dominated convergence theorem, the second term of
   \eqref{eq:treatment_T2} goes to $0$ as $d$ goes to infinity.
   Write $T_3^d = (h(\ell)/2) ( T^d_{3,1} - T^d_{3,2})$ where
   \begin{align*}
T^d_{3,1} &= \int_s^t  \left\{\phi''\left(
  X^d_{\floor{dr},1} \right) -  \phi''\left( Y^d_{r,1} \right)\right\} \rmd
  r \eqsp,\\
     T^d_{3,2} &= \int_s^t \left\{\Vdot\left( X^d_{\floor{dr},1} \right) \phi'\left(
  X^d_{\floor{dr},1}\right) - \Vdot \left( Y^d_{r,1} \right) \phi'\left( Y^d_{r,1} \right)\right\} \rmd
  r \eqsp.
     \end{align*}
     It is enough to show that $\PE[\left|T^d_{3,1}\right|]$ and $\PE[\left|T^d_{3,2}\right|]$ go to $0$ when $d$ goes to infinity to conclude the proof. By \eqref{eq:defY} and the mean value theorem, for all $r \in \ccint{s,t}$ there exists $\chi_r \in \ccint{X_{\floor{dr},1}^d , Y^d_{r,1}}$ such that
     \[
\phi''\left(
  X^d_{\floor{dr},1} \right) -  \phi''\left( Y^d_{r,1} \right) = \phi^{(3)} \left( \chi_r
\right) ( dr - \floor{dr}) (\ell/\sqrt{d}) Z^d_{\ceil{dr},1} \1_{\setAccept^d_{\ceil{dr}}} \eqsp.
     \]
     Since $\phi^{(3)}$ is bounded, it follows that $\lim_{d \to \plusinfty} \PE[
     |T^d_{3,1}| ] = 0$. On the other hand,
     \begin{multline*}
T_{3,2}^d = \int_s ^t \defEns{\Vdot\left( X^d_{\floor{dr},1} \right)  - \Vdot\left( Y^d_{r,1} \right) }\phi'\left(
  X^d_{\floor{dr},1}\right) \rmd r
\\
+ \int_s^t  \defEns{\phi'\left(
  X^d_{\floor{dr},1}\right) -  \phi'\left( Y^d_{r,1} \right) }\Vdot \left( Y^d_{r,1} \right)
\rmd r\eqsp.
      \end{multline*}
      Since $\phi'$ has a bounded support, by \Cref{assum:Vdot}, Fubini's theorem,
      and Lebesgue's dominated convergence
      theorem, the expectation of the absolute value of the first term goes to $0$ as
      $d$ goes to infinity. The second term is dealt with following the same steps as for
      $T^d_{3,1}$ and using \Cref{assum:diff:quadratic}\ref{assum:X6}.
\end{proof}

\begin{proof}[Proof of \Cref{theo:diffusion_limit_RMW}]
By \Cref{prop:tight}, \Cref{propo:reduction_martingale_problem} and
\Cref{propo:decomposition_martingale}, it is enough to prove that for all $\phi \in
C_c^\infty(\rset,\rset)$, $p\geq 1$, all $ 0 \leq t_1 \leq \dots \leq t_p \leq s \leq t$
and $g : \rset^p \to \rset$ bounded and continuous function,
\[
\lim_{\dim \to \plusinfty} \PE \left[( M^{\dim}_{\ceil{\dim t}}(\phi) - M^{\dim}_{\ceil{\dim s}}(\phi))g(Y^d_{t_1},\dots,Y^d_{t_p}) \right] = 0\eqsp,
\]
where for $n \geq 1$, $M^{\dim}_{n}(\phi)$ is defined in \eqref{eq:def:martingale}. But
this result is straightforward taking successively the conditional expectations with respect to $\mcf_{k}$, for $k =
\ceil{\dim t}, \dots, \ceil{\dim s}$.
\end{proof}


\section{Proofs of \Cref{sec:I}}
\label{sec:proofs:G}
\subsection{Proof of \Cref{theo:result_acceptance_rate_RWM:G}}
\label{proof:lem:approx_ratio:G}
\label{proof:theo:result_acceptance_rate_RWM:G}
The proof of this theorem follows the same steps as the the proof of \Cref{theo:result_acceptance_rate_RWM}. Note that $\xi_\theta$ and $\xi_0$, given by \eqref{eq:def_xitheta}, are well defined on $\I \cap \{ x \in \rset \ | \  x+\constSet \theta \in \I \}$. Let the function $\upsilon : \rset^2 \to \rset$ be defined for $x, \theta \in \rset$ by
 \begin{equation}
   \label{eq:def:upsilon}
   \upsilon(x,\theta) = \1_{\I}(x + \constSet \theta)\1_{\I}(x + (1-\constSet) \theta) \eqsp.
 \end{equation}
\begin{lemma}
\label{lem:DQM:G}
Assume \Cref{assum:diff:quadratic:G} holds. Then, there exists $C>0$ such that for all $\theta\in\rset$,
\[
\parenthese{\int_{\I} \left(\left\{\sqpdf_\theta(x)-\sqpdf_0(x)\right\} \upsilon(x,\theta)+ \theta \Vdot(x)\sqpdf_0(x)/2\right)^2 \rmd x}^{1/2} \le C|\theta|^{\beta}\eqsp.
\]
\end{lemma}
\begin{proof}
The proof follows as \Cref{lem:DQM} and is omitted.
\end{proof}
\begin{lemma}
\label{lem:integrated-DQM:G}
Assume that \Cref{assum:diff:quadratic:G} holds. Let $X$ be a random variable distributed according to $\pi$ and $Z$ be a standard Gaussian
random variable independent of $X$. Define
\begin{equation*}
    \Dcal_{\I} = \{ X + \constSet \ell d^{-1/2} Z \in \I\}\cap\{X + (1-\constSet) \ell d^{-1/2} Z \in \I\} \eqsp.
  \end{equation*}
  Then,
\begin{enumerate}[label=(\roman*)]
\item \label{lem:integrated-DQM-L2:G} $\lim_{d \to \plusinfty} d \ \left\| \1_{\Dcal_{\I}}\zeta^{\dim}(X,Z) +\ell Z\Vdot(X)/( 2\sqrt{\dim})\right\|_2^2 = 0$.
\item \label{lem:integrated-DQM-Lp:G} Let $p$ be given by  \Cref{assum:diff:quadratic:G}\ref{hyp:mean_square_deriv2:G}. Then,
  \begin{equation*}
    \lim_{d \to \plusinfty} \sqrt{d}\norm{\1_{\Dcal_{\I}}  \left\{V(X)-V(X+\ell Z /\sqrt{d})\right\}+\ell Z\Vdot(X)/\sqrt{d}}_p = 0 \eqsp.
  \end{equation*}
\item \label{lem:integrated-DQM-remainder:G} $\lim_{\dim \to \infty} \dim \left\| \1_{\Dcal_{\I}} \left( \log(1+\zeta_\dim(X,Z)) - \zeta^{\dim}(X,Z) + [\zeta^{\dim}]^2(X,Z)/2 \right)\right\|_1 = 0$,
\end{enumerate}
where $\zeta^{\dim}$ is given by \eqref{eq:def:zeta}.
\end{lemma}
\begin{proof}
Note by definition of $\zeta^{\dim}$ and $\sqpdf_\theta$ \eqref{eq:def_xitheta}, for $x \in \I$ and $x+ \constSet \ell d^{-1/2} z \in \I$,
\begin{equation}
 \label{eq:relationxi_zeta}
\zeta^{\dim}(x,z)=  \sqpdf_{  \ell
 z\dim^{-1/2}}(x)/\sqpdf_0(x)-1 \eqsp.
\end{equation}
Using \Cref{lem:DQM:G},
\begin{align*}
&\left\| \1_{\Dcal_{\I}}\zeta^{\dim}(X,Z) +\ell Z\Vdot(X)/( 2\sqrt{\dim})\right\|_2^2\\
&\qquad\qquad =
\PE\left[\int_{\I} \left(\upsilon(x,\ell Z \dim^{-1/2})\left\{\sqpdf_{\ell Z \dim^{-1/2}}(x)-\sqpdf_0(x)\right\} + \ell Z
   \Vdot(x)\sqpdf_0(x)/(2\sqrt{\dim})\right)^2 \rmd x\right] \\
&\qquad\qquad \leq C \ell^{2\beta}\dim^{-\beta}\PE\left[|Z|^{2\beta}\right]\eqsp.
\end{align*}
The proof of \ref{lem:integrated-DQM-L2:G} is completed using $\beta> 1$. For \ref{lem:integrated-DQM-Lp:G}, write for all $x \in \I$ and $x+\ell z d^{-1/2}z \in \I$, $\Delta V(x,z) = V(x) - V(x+\ell z d^{-1/2})$. By \Cref{assum:diff:quadratic:G}\ref{hyp:mean_square_deriv2:G}
\begin{align*}
\norm{\1_{\Dcal_{\I}}\Delta V(X,Z)+\ell Z\Vdot(X)/\sqrt{d}}_p^p & = \PE\left[\int_{\I} \left(\upsilon(x,\ell Z \dim^{-1/2})\Delta V(X,Z)+\ell Z\Vdot(x)/\sqrt{\dim}\right)^p\pi(x)\rmd x\right]\\
& \le C\ell^{\beta p}\dim^{-\beta p/2}\PE\left[\left|Z\right|^{\beta p}\right]
\end{align*}
and the proof of \ref{lem:integrated-DQM-Lp:G} follows from $\beta> 1$. For
\ref{lem:integrated-DQM-remainder:G}, note that for all $x >0$, $u \in \ccint{0,x}$, $\vert
(x-u)(1+u)^{-1} \vert \leq \abs{x}$, and the same inequality holds for $x \in
\ocint{-1,0}$ and $u \in \ccint{x,0}$. Then by \eqref{eq:def_R} and
\eqref{eq:dev_taylor_log}, for all $x>-1$,
\[
\left|\log(1+x)-x+x^2/2\right|= \abs{R(x)} \le x^2\left|\log(1+x)\right| \eqsp.
\]
Then by \eqref{eq:relationxi_zeta}, for $x \in \I$ and $x+ \ell d^{-1/2} z \in \I$,
\begin{align*}
&\left|\log(1+\zeta_\dim(x,z)) - \zeta^{\dim}(x,z) + [\zeta^{\dim}]^2(x,z)/2 \right| \\
&\qquad \qquad \le \left(\sqpdf_{\ell z\dim^{-1/2}}(x)/\sqpdf_0(x)-1\right)^2\left|\log(\sqpdf_{\ell z\dim^{-1/2}}(x)/\sqpdf_0(x))\right|\eqsp,\\
&\qquad  \qquad \le \left(\sqpdf_{\ell z\dim^{-1/2}}(x)/\sqpdf_0(x)-1\right)^2\left|V(x+\ell z\dim^{-1/2}) -V(x)\right|/2\eqsp.
\end{align*}
Since for all $x\in\rset$, $|\exp(x)-1|\le |x|(\exp(x)+1)$, this yields,
\begin{multline*}
\left|\log(1+\zeta_\dim(x,z)) - \zeta^{\dim}(x,z) + [\zeta^{\dim}]^2(x,z)/2 \right|\\
\le \left|V(x+\ell z\dim^{-1/2}) -V(x)\right|^3\left(\exp\left(V(x)-V(x+\ell z\dim^{-1/2})\right)+1\right)/4\eqsp.
\end{multline*}
Therefore,
\[
\int_{\I}\upsilon(x,\ell z d^{-1/2}) \left|\log(1+\zeta_\dim(x,z)) - \zeta^{\dim}(x,z) + [\zeta^{\dim}]^2(x,z)/2 \right| \pi(x)\rmd x\le (I_1 +I_2)/4 \eqsp,
\]
where
\begin{align*}
I_1 &= \int_{\I } \upsilon(x,\ell z d^{-1/2}) \left|V(x+\ell z\dim^{-1/2}) -V(x)\right|^3\pi(x)\rmd x\\
I_2 &= \int_{\I } \upsilon(x,\ell z d^{-1/2}) \left|V(x+\ell z\dim^{-1/2}) -V(x)\right|^3\pi(x+\ell z\dim^{-1/2})\rmd x\eqsp.
\end{align*}
By H\"{o}lder's inequality, a change of variable and using \Cref{assum:diff:quadratic:G}\ref{hyp:mean_square_deriv2:G},
\[
I_1+I_2\le C\left(\left|\ell z\dim^{-1/2}\right|^3 \left(\int_{\I}  \left|\Vdot(x)\right|^4\pi(x)\rmd x\right)^{3/4}+ \left|\ell z\dim^{-1/2}\right|^{3\beta}\right)\eqsp.
\]
The proof follows from \Cref{assum:diff:quadratic:G}\ref{assum:X6:G} and $\beta>1$.

\end{proof}

For ease of notation, write for all $d \geq 1$ and  $i,j \in \{ 1, \ldots, d \}$,
\begin{align}
\nonumber
  \Dcal^{d}_{\I,j} &= \defEns{ X_j^\dim + \constSet \ell d^{-1/2} Z_j^\dim \in \I} \cap \defEns{ X_j^\dim +(1- \constSet) \ell d^{-1/2} Z_j^\dim \in \I} \eqsp, \\
\label{eq:def_dcal_I}
 \ \Dcal^{d}_{\I,i:j} &= \bigcap_{k=i}^j \Dcal^{d}_{\I,k}  \eqsp.
\end{align}
\begin{lemma}
  \label{lem:interm_prop_approx_ratio:G}
Assume that \Cref{assum:diff:quadratic:G} holds. For all $d \geq 1$, let $X^\dim$ be distributed according to $\pi^\dim$, and $Z^\dim$ be
$d$-dimensional Gaussian random variable independent of $X^\dim$. Then, $\lim_{\dim\to\plusinfty}\Jrm^{\dim}_{\I}= 0$ where
\[
\Jrm^{\dim}_{\I}=\left\|   \1_{\Dcal^\dim_{\I,2:d}} \sum_{i=2}^{\dim}
\left\{  \left( \Delta V^{\dim}_{i}+\frac{\ell
  Z_{i}^{\dim}}{\sqrt{\dim}}\Vdot(X_{i}^{\dim})\right)-2\PE\left[\1_{\Dcal^\dim_{\I,i}}\zeta^{\dim}(X_{i}^{\dim},Z_{i}^{\dim})\right]
+  \frac{\ell^2 }{4\dim}\Vdot^2(X_{i}^{\dim}) \right\} \right\|_1\eqsp.
\]
\end{lemma}
\begin{proof}
The proof follows the same lines as the proof of \Cref{lem:interm_prop_approx_ratio} and is omitted.
\end{proof}

Define for all $d \geq 1$,
\begin{multline*}
\Erm^{\dim}_{\I} = \PE\left[\left(Z_{1}^d\right)^2\left|  \1_{ \Dcal^d_{\I,1:d}} 1 \wedge \exp\left\{\sum_{i=1}^{\dim} \Delta V^{\dim}_{i}\right\} \right. \right. \\
\left. \left.  - 1 \wedge \exp\left\{ -\ell \dim^{-1/2}Z_{1}^{\dim}\Vdot(X_{1}^{\dim}) +\sum_{i=2}^{\dim} b_{\I}^{\dim}(X_i^{\dim},Z_i^{\dim}) \right\} \right| \right]\eqsp,
  \end{multline*}
where $\Delta V^{\dim}_{i}$ is given by \eqref{eq:Delta_V}, for all $x \in \I$, $z \in \rset$,
\begin{equation}
   \label{eq:bdi:G}
b^{\dim}_{\I}(x,z)= -\frac{\ell
  z}{\sqrt{\dim}}\Vdot(x)+2\PE\left[\1_{\Dcal^d_{\I,1}}\zeta^{\dim}(X_1^{\dim},Z_1^{\dim})\right]
-  \frac{\ell^2 }{4\dim}\Vdot^2(x) \eqsp,
\end{equation}
and $\zeta^d$ is given by \eqref{eq:def:zeta}.

\begin{proposition}
\label{lem:approx_ratio:G}
Assume \Cref{assum:diff:quadratic:G} holds. Let $X^{\dim}$ be a random variable distributed
according to $\target^{\dim}$ and $Z^{\dim}$ be a zero-mean standard Gaussian random variable, independent of $X$. Then $\lim_{\dim\to\plusinfty} \Erm^{\dim}_\I = 0$.
\end{proposition}
\begin{proof}
Let $\Lambda^d = - \ell \dim^{-1/2} Z_{1}^{d}\Vdot(X_{1}^{d}) + \sum_{i=2}^d \Delta V^{\dim}_{i}$. By the triangle inequality,
$\Erm^{\dim}\leq \Erm^{\dim}_1 + \Erm^{\dim}_2 + \Erm^{\dim}_3$ where
\begin{align*}
\Erm^{\dim}_{1,\I} &= \PE\left[ \left(Z_{1}^d\right)^2  \1_{\Dcal^d_{\I,1:d}} \left|  1 \wedge \exp\left\{\sum_{i=1}^d \Delta V^{\dim}_{i}\right\} - 1 \wedge \exp\left\{\Lambda^d\right\} \right| \right]\eqsp,\\
\Erm^{\dim}_{2,\I} &= \PE\left[ \left(Z_{1}^d\right)^2  \1_{\Dcal^d_{\I,2:d}} \left|1 \wedge \exp\left\{\Lambda^d\right\}  - 1 \wedge \exp\left\{ -\ell \dim^{-1/2} Z_1^{d}\Vdot(X_{1}^{d}) +\sum_{i=2}^d b^d(X_i^{\dim},Z_i^{\dim})\right\} \right| \right] \eqsp,\\
\Erm^{\dim}_{3,\I} &= \PE\left[ \left(Z_{1}^d\right)^2  \1_{\left(\Dcal^d_{\I,2:d}\right)^c} 1 \wedge \exp\left\{ -\ell \dim^{-1/2} Z_1^{d}\Vdot(X_{1}^{d}) +\sum_{i=2}^d b^d(X_i^{\dim},Z_i^{\dim})\right\}\right] \eqsp,
\end{align*}
Since $t \mapsto 1 \wedge \rme^t$ is $1$-Lipschitz, by the Cauchy-Schwarz inequality we get
\[
\Erm^{\dim}_{1,\I} \leq \PE\left[ \left(Z_{1}^d\right)^2  \1_{\Dcal^d_{\I,1}} \left| \Delta V^{\dim}_{1}+\ell \dim^{-1/2} Z_{1}^{d}\Vdot(X_{1}^{d}) \right| \right]\le \|  Z_1^{\dim} \|_4^2 \left\|\1_{\Dcal^d_{\I,1}}\Delta V^{\dim}_1 + \ell \dim^{-1/2} Z_{1}^{d}\Vdot(X_{1}^{d}) \right\|_2  \eqsp.
\]
By \Cref{lem:integrated-DQM}\ref{lem:integrated-DQM-Lp}, $\Erm^{\dim}_{1,\I}$ goes to $0$ as $d$ goes to $\plusinfty$. Using again that $t \mapsto 1 \wedge \rme^t$ is $1$-Lipschitz and \Cref{lem:interm_prop_approx_ratio:G}, $\Erm^{\dim}_{2,\I}$ goes to $0$ as well. Note that, as $Z_1^d$ and  $\1_{\left(\Dcal^d_{\I,2:d}\right)^c}$ are independent, by \eqref{eq:application-hyp},
\[
\Erm^{\dim}_{3,\I} \le d \PP \left(\left\{\Dcal^d_{\I,1}\right\}^c\right) \le Cd^{1-\gamma/2} \eqsp.
\]
Therefore, $\Erm^{\dim}_{3,\I}$  goes to $0$ as $d$ goes to $\plusinfty$ by \Cref{assum:diff:quadratic:G}\ref{assum:proba:G}.
\end{proof}

\begin{lemma}
\label{lem:mean:zeta:G}
Assume \Cref{assum:diff:quadratic:G} holds. For all $d \in \nset^*$, let
$X^{\dim}$ be a random variable distributed according to $\target^{\dim}$ and $Z^{\dim}$
be a standard Gaussian random variable in $\rset^d$, independent of $X$. Then,
\[
\lim_{d \to \plusinfty}2\dim\,\PE\left[\1_{\Dcal_{\I,1}^d}\zeta^{\dim}(X^{\dim}_1,Z^{\dim}_1)\right] = -\frac{\ell^2}{4}I\eqsp,
\]
where $I$ is defined in \eqref{eq:translation-score} and $\zeta^{\dim}$ in \eqref{eq:def:zeta}.
\end{lemma}
\begin{proof}
Noting that for all $\theta\in \mathbb{R}$,
\[
\int_{\I} \1_{\I}(x+\constSet \theta) \1_{\I}(x+(1-\constSet) \theta) \pi(x+\theta)\rmd x = \int_{\I} \1_{\I}(x+(\constSet-1) \theta)\1_{\I}(x-\constSet\theta)\pi(x)\rmd x\eqsp.
\]
the proof follows the same steps as the the proof of \Cref{lem:mean:zeta} and is omitted.
\end{proof}

\begin{proof}[Proof of \Cref{theo:result_acceptance_rate_RWM:G}]
The proof follows the same lines as the proof of \Cref{theo:result_acceptance_rate_RWM} and is therefore omitted.
\end{proof}

\subsection{Proof of \Cref{prop:tight:G}}
\label{sec:proof:tightness:G}
As for the proof of \Cref{prop:tight}, the proof follows from \Cref{lem:kolmo:G}.

\begin{lemma}
\label{lem:kolmo:G}
Assume \Cref{assum:diff:quadratic:G}. Then, there exists $C>0$ such that, for all $0\le k_1<k_2$,
\[
\PE\left[\left(X_{k_2,1}^{\dim} - X_{k_1,1}^{\dim}\right)^4\right] \le C \sum_{p=2}^4 \frac{(k_2-k_1)^p}{\dim^p} \eqsp.
\]
\end{lemma}

\begin{proof}
We use the same decomposition of $\PE[(X_{k_2,1}^{\dim} - X_{k_1,1}^{\dim})^4]$ as in the proof of \Cref{lem:kolmo} so that we only need to upper bound the following term:
\[
d^{-2}\PE\left[\left(\sum_{k=k_1+1}^{k_2}Z_{k,1}^{\dim}\1_{\left(\setAccept_{k}^{\dim}\right)^c}\right)^4\right] = d^{-2}\sum \PE\left[\prod_{i=1}^4Z_{m_i,1}^{\dim}\1_{\left(\setAccept_{m_i}^{\dim}\right)^c}\right]\eqsp,
\]
where the sum is over all the quadruplets $(m_p)_{p=1}^4$ satisfying $m_p \in \{k_1+1,\dots,k_2\}$, $p=1,\dots,4$.
Let $(m_1,m_2,m_3,m_4)\in\{k_1+1,\dots,k_2\}^4$  and  $(\tilde{X}_k^{\dim})_{k\ge 0}$ be defined as:
\[
\tilde{X}_0^{\dim} = X_0^{\dim} \quad \mbox{and} \quad \tilde{X}_{k+1}^{\dim} = \tilde{X}_{k}^{\dim} + \1_{k\notin\left\{m_1-1,m_2-1,m_3-1,m_4-1\right\}}\ell\dim^{-1/2}Z_{k+1}^{\dim}\1_{\tilde{\setAccept}_{k+1}^{\dim}}\eqsp,
\]
where for all $k\ge 0$ and all $1\le i \le \dim$,
\begin{align*}
\tilde{\setAccept}_{k+1}^{\dim} &= \left\lbrace U_{k+1} \leq \exp \left(\sum_{i=1}^{\dim} \Delta \tilde{V}^{\dim}_{k,i}\right) \right\rbrace\\
 \Delta \tilde{V}^{\dim}_{k,i} &= V\left(\tilde{X}_{k,i}^{\dim}\right) - V\left(\tilde{X}_{k,i}^{\dim} + \ell \dim^{-1/2} Z^{\dim}_{k+1,i}\right)\eqsp.
\end{align*}
Define, for all $k_1+1\le k\le k_2$, $1\le i,j\le d$,
\begin{align*}
\tilde{\Dcal}^{d,k}_{\I,j} &= \defEns{ \tilde{X}_{k,j}^d + \constSet \ell d^{-1/2} Z^d_{k+1,j} \in \I} \cap \defEns{ \tilde{X}_{k,j}^d + (1-\constSet)\ell d^{-1/2} Z^d_{k+1,j} \in \I}\eqsp,  \\
 \tilde{\Dcal}^{d,k}_{\I,i:j} &= \bigcap_{\ell=i}^j \tilde{\Dcal}^{d,k}_{\I,\ell}\eqsp.
\end{align*}
Note that by convention $V(x) = -\infty$ for all $x\notin \I$, $\tilde{\setAccept}_{k+1}^{\dim} \subset \tilde{\Dcal}^{d,k}_{\I,1:d}$ so that $\left(\tilde{\setAccept}_{k+1}^{\dim}\right)^c$ may be written $\left(\tilde{\setAccept}_{k+1}^{\dim}\right)^c = \left(\tilde{\Dcal}^{d,k}_{\I,1:d}\right)^c\bigcup \left(\left(\tilde{\setAccept}_{k+1}^{\dim}\right)^c\cap \tilde{\Dcal}^{d,k}_{\I,1:d} \right)$. Let $\mcf$ be the $\sigma$-field generated by $\left(\tilde{X}^d_k\right)_{k\ge 0}$. Consider the case $\#\{m_1,\ldots,m_4\}=4$. The case $\#\{m_1,\ldots,m_4\}=3$ is dealt with similarly and the two other cases follow the same lines as the proof of \Cref{lem:kolmo:G}. As $\left\{\left(U_{m_j},Z_{m_j,1}^{\dim}, \cdots,Z_{m_j,\dim}^{\dim} \right)\right\}_{1\le j\le 4}$ are independent conditionally to $\mcf$,
\[
\PE\left[\prod_{j=1}^4Z_{m_j,1}^{\dim}\1_{\left(\setAccept_{m_j}^{\dim}\right)^c}\middle| \mcf\right] = \prod_{j=1}^4\left\{\PE\left[\1_{\left(\tilde{\Dcal}^{d,m_{j-1}}_{\I,1:d}\right)^c}Z_{m_j,1}^{\dim}\middle| \mcf\right] + \PE\left[\1_{\tilde{\Dcal}^{d,m_{j-1}}_{\I,1:d}}\1_{\left(\tilde{\setAccept}_{m_j}^{\dim}\right)^c}Z_{m_j,1}^{\dim}\middle| \mcf\right]\right\}\eqsp.
\]
As $U_{m_j}$ is independent of $(Z_{m_j,1}^{\dim}, \cdots,Z_{m_j,\dim}^{\dim})$ conditionally to $\mcf$, the second term may be written:
\[
\PE\left[\1_{\tilde{\Dcal}^{d,m_{j-1}}_{\I,1:d}}\1_{\left(\tilde{\setAccept}_{m_j}^{\dim}\right)^c}Z_{m_j,1}^{\dim}\middle| \mcf\right] = \PE\left[\1_{\tilde{\Dcal}^{d,m_{j-1}}_{\I,1:d}}Z_{m_j,1}^{\dim}\left(1-\mathrm{exp}\left\{\sum_{i=1}^{\dim} \Delta \tilde{V}^{\dim}_{m_j-1,i}\right\}\right)_+\middle| \mcf\right]\eqsp.
\]
Since the function $x\mapsto \left(1-\mathrm{e}^x\right)_+$ is 1-Lipschitz, on $\tilde{\Dcal}^{d,m_{j-1}}_{\I,1:d}$
\begin{equation*}
\left|\left(1-\mathrm{exp}\left\{\sum_{i=1}^{\dim} \Delta \tilde{V}^{\dim}_{m_j-1,i}\right\}\right)_+ -\Theta_{m_j} \right|
\le\left|\Delta \tilde{V}^{\dim}_{m_j-1,1} + \ell \dim^{-1/2} \Vdot(\tilde{X}_{m_j-1,1}^{\dim})Z^{\dim}_{m_j,1}\right|\eqsp,
\end{equation*}
where $\Theta_{m_j} = (1-\mathrm{exp}\{-\ell \dim^{-1/2}\Vdot(\tilde{X}_{m_j-1,1}^{\dim})Z^{\dim}_{m_j,1} + \sum_{i=2}^{\dim} \Delta \tilde{V}^{\dim}_{m_j-1,i}\})_+$.
Then,
\[
\left|\PE\left[\1_{\tilde{\Dcal}^{d,m_{j-1}}_{\I,1:d}}Z_{m_j,1}^{\dim}\left(1-\mathrm{exp}\left\{\sum_{i=1}^{\dim} \Delta \tilde{V}^{\dim}_{m_j-1,i}\right\}\right)_+\middle| \mcf\right]\right| \le A_{m_j}^{\dim} + B_{m_j}^{\dim}\eqsp,
\]
where
\begin{align*}
A_{m_j}^{\dim}&= \PE\left[\left|Z_{m_j,1}^{\dim}\right|\left|\1_{\tilde{\Dcal}^{d,m_{j-1}}_{\I,1}}\Delta \tilde{V}^{\dim}_{m_j-1,1} + \ell\dim^{-1/2}\Vdot(\tilde{X}_{m_j-1,1}^{\dim})Z^{\dim}_{m_j,1}\right|\middle|\mcf\right] \eqsp,\\
B_{m_j}^{\dim}&= \left|\PE\left[\1_{\tilde{\Dcal}^{d,m_{j-1}}_{\I,2:d}}Z_{m_j,1}^{\dim} \Theta_{m_j}\middle|\mcf\right]\right|\eqsp.
\end{align*}
By Jensen inequality,
\begin{multline*}
\left|\PE\left[\prod_{j=1}^4Z_{m_j,1}^{\dim}\1_{\left(\setAccept_{m_j}^{\dim}\right)^c}\right]\right|
\leq  \PE\left[\prod_{j=1}^4\left\{\PE\left[\1_{\left(\tilde{\Dcal}^{d,m_{j-1}}_{\I,1:d}\right)^c}|Z_{m_j,1}^{\dim}|\middle| \mcf\right] + A_{m_j}^{\dim} + B_{m_j}^{\dim}\right\}\right]\eqsp,\\
 \leq C\PE\left[\sum_{j=1}^4\PE\left[\1_{\left(\tilde{\Dcal}^{d,m_{j-1}}_{\I,1:d}\right)^c}|Z_{m_j,1}^{\dim}|^4\middle| \mcf\right] + \left(A_{m_j}^{\dim}\right)^4 + \left(B_{m_j}^{\dim}\right)^4\right]\eqsp,
\end{multline*}
By \Cref{assum:diff:quadratic:G}\ref{assum:proba:G} and Holder's inequality applied with $\alpha= 1/(1-2/\gamma)>1$, for all $1\le j \le 4$,
\begin{align*}
\PE\left[\1_{\left(\tilde{\Dcal}^{d,m_{j-1}}_{\I,1:d}\right)^c}|Z_{m_j,1}^{\dim}|^4\right]&\le \PE\left[\1_{\left(\tilde{\Dcal}^{d,m_{j-1}}_{\I,1}\right)^c}|Z_{m_j,1}^{\dim}|^4\right] +  \sum_{i=2}^d \PE\left[\1_{\left(\tilde{\Dcal}^{d,m_{j-1}}_{\I,i}\right)^c}\right]\eqsp,\\
&\le \PE\left[|Z_{m_j,1}^{\dim}|^{4\alpha/(\alpha-1)}\right]^{(\alpha-1)/\alpha}d^{-\gamma/(2\alpha)} + d^{1-\gamma/2}\eqsp,\\
&\le Cd^{1-\gamma/2}\eqsp.
\end{align*}
By \Cref{lem:integrated-DQM:G}\ref{lem:integrated-DQM-Lp:G} and the Holder's inequality, there exists $C>0$ such that $\PE\left[\left(A_{m_j}^{\dim}\right)^4\right]\le C\dim^{-2}$. On the other hand, by \cite[Lemma~6]{jourdain:lelievre:miasojedow:2015} since $Z^d_{m_j,1}$ is independent of $\mcf$,
\[
B_{m_j}^{\dim} = \left|\PE\left[\1_{\tilde{\Dcal}^{d,m_{j-1}}_{\I,2:d}} \ell \dim^{-1/2} \Vdot(\tilde{X}_{m_j-1,1}^{\dim})\G\left(\ell^2\dim^{-1}\Vdot(\tilde{X}_{m_j-1,1}^{\dim})^2,-2\sum_{i=2}^{\dim} \Delta \tilde{V}^{\dim}_{m_j-1,i}\right)\middle|\mcf\right]\right|\eqsp,
\]
where the function $\G$ is defined in \eqref{eq:defG}. By \Cref{assum:diff:quadratic:G}\ref{assum:X6:G} and since $\G$ is bounded, $\PE[(B_{m_j}^{\dim})^4]\le C\dim^{-2}$. Since $\gamma\ge 6$ in \Cref{assum:diff:quadratic:G}\ref{assum:proba:G}, $|\PE[\prod_{j=1}^4Z_{m_j,1}^{\dim}\1_{(\setAccept_{m_j}^{\dim})^c}]|\le C\dim^{-2}$, showing that
\begin{equation}
\label{eq:card4}
\sum_{(m_1,m_2,m_3,m_4)\in\mathcal{I}_4}\left|\PE\left[\prod_{i=1}^4Z_{m_i,1}^{\dim}\1_{\left(\setAccept_{m_i}^{\dim}\right)^c}\right]\right|\le C\dim^{-2}{ k_2-k_1 \choose 4}\eqsp.
\end{equation}
\end{proof}


\subsection{Proof of \Cref{propo:reduction_martingale_problem:G}}
\label{sec:reduction_martingale_problem:G}
\begin{lemma}
\label{lem:law_point_limit:G}
Assume that \Cref{assum:diff:quadratic:G} holds. Let $\mu$ be a  limit point of the sequence of laws $(\mu_{\dim})_{\dim\ge 1}$ of $\defEns{(Y_{t,1}^{\dim})_{t\geq 0}, \ \dim \in
  \Nset^*}$. Then for all $t \geq 0$, the pushforward measure of $\mu$ by $W_t$ is $\pdf$.
   \end{lemma}
\begin{proof}
The proof is the same as in \Cref{lem:law_point_limit} and is omitted.
\end{proof}
We preface the proof by a lemma which provides a condition to verify that any limit point $\mu$ of $(\mu_d)_{ d \geq 1}$ is a solution to the local martingale problem associated with \eqref{eq:langevin_limit}.
\begin{lemma}
\label{propo:reduction_martingale_problem_01:G}
Assume \Cref{assum:diff:quadratic:G}. Let $\mu$ be a limit point of the sequence $(\mu_d)_{d \geq 1}$. If for all  $\phi \in C^{\infty}_c(\I,\rset)$,
the process
$(\phi(W_t) - \phi(W_0) - \int_{0}^t \Lrm \phi(W_u) \rmd u)_{ t \geq 0} $
is a  martingale with respect to $\mu$ and the filtration
$(\canonicalFiltration_t)_{t \geq 0}$, then $\mu$ solves the local martingale problem associated with \eqref{eq:langevin_limit}.
\end{lemma}
\begin{proof}
  As for all $t \geq 0$ and $d \geq 1$, $Y_{t,1}^d \in \I$, for all $d
  \geq 1$ $\mu^{d}(C(\rset_+, \Iclosed)) = 1$. Since $C(\rset_+,
  \Iclosed)$ is closed in $\Wienerspace$, we have by the Portmanteau
  theorem, $\mu(C(\rset_+, \Iclosed)) =1$. Therefore, we only need to
  prove that for all $\psi \in C^\infty(\Iclosed,\rset)$, the process
  $(\psi(W_t) - \psi(W_0) - \int_{0}^t \Lrm \psi(W_u) \rmd u)_{ t \geq 0} $ is a
  local martingale with respect to $\mu$ and the filtration
  $(\canonicalFiltration_t)_{t \geq 0}$. Let $\psi \in
  C^\infty(\Iclosed,\rset)$.

  Suppose first that for all $\varpi \in
  C^\infty_c(\Iclosed,\rset)$, $(\varpi(W_t) - \varpi(W_t) - \int_{0}^t \Lrm
  \varpi(W_u) \rmd u)_{ t \geq 0} $ is a martingale. Then,
  consider the sequence of stopping time defined for $k \geq 1$ by
  $\tau_k = \inf \{ t \geq 0 \ | \ \abs{W_t} \geq k\}$ and a sequence
$(\varpi_k)_{k \geq 0}$ in $C^\infty_c(\Iclosed,\rset)$ satisfying:
\begin{enumerate}
\item for all $k \geq 1$ and all $x \in\Iclosed \cap \ccint{-k , k }$, $\varpi_k(x) = \psi(x)$,
\item $\lim_{k\to \plusinfty} \varpi_k = \psi$ in $C^\infty(\Iclosed,\rset)$.
\end{enumerate}
Since for all $k \geq 1$,
\begin{multline*}
\left(\psi(W_{t\wedge \tau_k}) - \psi(W_{0}) - \int_{0}^{t\wedge \tau_k} \Lrm \psi(W_u) \rmd u\right)_{ t \geq 0}  \\ = \left(\varpi_k(W_{t\wedge \tau_k}) - \varpi_k(W_{0})-  \int_{0}^{t\wedge \tau_k} \Lrm \varpi_k(W_u) \rmd u\right)_{ t \geq 0}
\end{multline*}
and the sequence $(\tau_k)_{k \geq 1}$ goes to $\plusinfty$ as $k$
goes to $\plusinfty$ almost surely, it follows that $(\psi(W_t) - \psi(W_0) -
\int_{0}^t \Lrm \psi(W_u) \rmd u)_{ t \geq 0} $ is a local martingale
with respect to $\mu$ and the filtration $(\canonicalFiltration_t)_{t
  \geq 0}$. It remains to show that for all $\varpi \in
C^\infty_c(\Iclosed,\rset)$, $(\varpi(W_t) - \varpi(W_0) - \int_{0}^t \Lrm \varpi(W_u)
\rmd u)_{ t \geq 0} $ is a martingale under the assumption of the
proposition.  We only need to prove that for all $\varpi \in C^\infty_c(\Iclosed,\rset)$,
$0 \leq s\leq t$, $m \in \nset^*$, $g : \rset^m \to \rset$ bounded and continuous, and $0 \leq t_1 \leq \dots \leq t_m \leq s \leq t$:
 \begin{equation}
 \label{eq:reduction_martingale_problem:10:G}
  \PE^{\mu}\left[ \left(\varpi\left(W_t\right) - \varpi\left( W_s \right) - \int_s^t \Lrm \varpi \left(W_u \right) \rmd u  \right) g\left(W_{t_1},\dots,W_{t_m} \right) \right] =0 \eqsp.
\end{equation}
Let $(\phi_k)_{k \geq 0}$ be a sequence of functions in
$C^\infty_c(\I,\rset)$ and converging to $\varpi$ in
$C^\infty_c(\Iclosed,\rset)$. First note that for all $u \in \ccint{s,t}$, $\mu$-almost everywhere,
\begin{equation}
\label{eq:reduction_martingale_problem:0_2:G}
 \lim_{ k \to \plusinfty} \phi_k(W_u) = \varpi(W_u)\eqsp.
\end{equation}
By \Cref{lem:law_point_limit:G}, for all $u \in \ccint{s,t}$ the pushforward measure of $\mu$ by $W_u$ has density $\pdf$ with respect to the Lebesgue measure and  $\mu$-almost everywhere, $ \lim_{k \to
  \plusinfty} \Lrm \phi_k(W_u) = \Lrm \varpi(W_u)$. On the other hand, there exists $C \geq
0$ such that for all $k \geq 0$, $|\Lrm \phi_k(W_u)| \leq
C(1+|\Vdot(W_u)|)$. Then,
\begin{align*}
\PE^{\mu}\left[\int_s^t\left(1+|\Vdot(W_u)|\right)\rmd u\right] 
&\le (t-s) + \int_s^t \PE^{\mu}\left[|\Vdot(W_u)|\right]\rmd u \\
&\le (t-s)\left(1 + \int_{\I} |\Vdot(x)|\pdf(x)\rmd x\right)\eqsp.
\end{align*}
Therefore, $\mu$-almost everywhere by \Cref{assum:diff:quadratic:G}\ref{assum:X6:G} and the Lebesgue dominated convergence theorem, we get
\begin{equation}
\label{eq:reduction_martingale_problem:0_3:G}
\lim_{k \to \plusinfty } \int_s^t  \Lrm \phi_k(W_u) \rmd u = \int_s^t \Lrm \varpi(W_u) \rmd u\eqsp.
\end{equation}
Therefore, \eqref{eq:reduction_martingale_problem:10:G} follows from  \eqref{eq:reduction_martingale_problem:0_2:G} and \eqref{eq:reduction_martingale_problem:0_3:G}, using again the Lebesgue dominated convergence theorem and \Cref{assum:diff:quadratic:G}\ref{assum:X6:G}.
\end{proof}

\begin{proof}[Proof of \Cref{propo:reduction_martingale_problem:G}]
  Let $\mu$ be a limit point of $(\mu_{\dim})_{\dim\ge 1}$. By
  \Cref{propo:reduction_martingale_problem_01:G}, we only need to
  prove that for all $\phi \in C^{\infty}_c(\I,\rset)$, the
  process $(\phi(W_t) -\phi(W_0)- \int_{0}^t \Lrm \phi(W_u) \rmd u)_{ t \geq 0} $
  is a martingale with respect to $\mu$ and the filtration
  $(\canonicalFiltration_t)_{t \geq 0}$. Then, the proof follows the same line as the proof of \Cref{propo:reduction_martingale_problem} and is omitted.
\end{proof}


\subsection{Proof of \Cref{theo:diffusion_limit_RMW:G}}
\label{sec:proof:weaklimit:G}
\begin{lemma}
\label{lem:approx_first_term_decomposition_martingale:G}
Assume \Cref{assum:diff:quadratic:G} holds. Let
$X^d$ be distributed according to $\target^{d}$ and $Z^d$ be a
$d$-dimensional standard Gaussian random variable, independent of $X^d$. Then, $\lim_{d
  \to \plusinfty} \Erm^{\dim} = 0$, where
\begin{equation*}
\Erm^{\dim}=  \expe{\abs{\Vdot(X_1^d)\1_{\Dcal_{\I,2:d}^{d}}\defEns{\mathcal{G}\parenthese{\ell^2\Vdot(X_1^d)^2/d, 2 \barY_d}-\mathcal{G}\parenthese{\ell^2\Vdot(X_1^d)^2/d, 2 \barX_d}}} }  \eqsp,
\end{equation*}
where $\barY_d = \sum_{i=2}^d \Delta V_i^d$, $\Delta V_i^d$ and $\Dcal_{\I,2:d}^{d}$ are given by \eqref{eq:Delta_V} and \eqref{eq:def_dcal_I} and $\barX_d = \sum_{i=2}^d b_{\I,i}^d$, $b_{\I,i}^d= b_{\I}^d(X_i^d,Z_i^d)$ with $b_{\I}^d$ given by \eqref{eq:bdi:G}.
\end{lemma}

\begin{proof}
Set for all $d \geq 1$, $\barY_d = \sum_{i=2}^d \Delta V_i^d$ and $\barX_d = \sum_{i=2}^d b_{\I,i}^d$. By definition of $b_{\I}^d$ \eqref{eq:bdi:G},  $\barX_d$ may be expressed as $\barX_d = \sigma_d \barS_d +\mu_d$, where
\begin{align*}
\mu_d &= 2 (d-1) \expe{\1_{\Dcal_{\I,1}^d}\zeta^{\dim}(X_1^{\dim},Z_1^{\dim})} - \frac{\ell^2(d-1)}{4d} \expe{\Vdot(X_1^{\dim})^2}\eqsp, \\
\sigma^2_d &= \ell^2 \expe{\Vdot(X_1^{\dim})^2} + \frac{\ell^4}{16d} \expe{\parenthese{\Vdot(X_1^{\dim})^2 - \expe{\Vdot(X_1^{\dim})^2}}^2}\eqsp,  \\
\barS_d &= (\sqrt{d} \sigma_d)^{-1}\sum_{i=2}^d \beta_i^d\eqsp, \\
\beta_i^d &= -\ell Z_i^d \Vdot(X_i^d) - \frac{\ell^2}{4\sqrt{\dim}}\parenthese{\Vdot(X_i^d)^2 - \expe{\Vdot(X_i^d)^2}} \eqsp.
\end{align*}
By \Cref{assum:diff:quadratic:G}\ref{assum:X6:G} the Berry-Essen Theorem \cite[Theorem 5.7]{petrov:1995} can be applied to $\barS_d$. Then, there exists a universal constant $C$ such that for all $\dim>0$,
\[
\sup_{x \in \rset} \abs{\proba{\sqrt{\frac{d}{d-1}}\barS_d \leq x}-\Phi(x)} \leq C/ \sqrt{d} \eqsp.
\]
It follows, with $\tilde{\sigma}_d ^2= (d-1)\sigma_d^2/d$, that
\[
\sup_{x \in \rset} \abs{\proba{\barX_d \leq x}-\Phi((x-\mu_d)/\tilde{\sigma}_d)} \leq C/ \sqrt{d} \eqsp.
\]
By this result and \eqref{eq:G_lip_in_b}, \Cref{lem:fun_nearly_lip} can be applied to obtain a constant $C \geq 0$, independent of $\dim$, such that:
\begin{multline*}
\expe{\1_{\Dcal_{\I,2:d}^{d}}\abs{\mathcal{G}\parenthese{\ell^2\Vdot(X_1^d)^2/d, 2 \barY_d}-\mathcal{G}\parenthese{\ell^2\Vdot(X_1^d)^2/d, 2 \barX_d} \sachant{X_1^d}}} \\
\leq  C \left( \expe{\1_{\Dcal_{\I,2:d}^{d}}\abs{\barX_d-\barY_d}} + d^{-1/2} + \sqrt{2\expe{\1_{\Dcal_{\I,2:d}^{d}}\abs{\barX_d-\barY_d}}(2 \pi \tilde{\sigma}_d^2)^{-1/2}} \right. \\
\left.+  \sqrt{\ell|\Vdot(X_1^d)|/( 2 \pi d^{1/2} \tilde{\sigma}_d^2 )} \right) \eqsp.
\end{multline*}
Using this result, we have
\begin{align}
\label{eq:approx_first_term_decomp_martingale_3:G}
&\Erm^{\dim}\leq  C \left\{
\ell^{1/2}\expe{|\Vdot(X_1^d)|^{3/2}} ( 2 \pi d^{1/2} \tilde{\sigma}^2_d )^{-1/2} +  \expe{|\Vdot(X_1^d)|}
\right.   
\\
\nonumber
&\qquad\qquad \left. \times 
\parenthese{ \expe{\1_{\Dcal_{\I,2:d}^{d}}\abs{\barX_d-\barY_d}} + d^{-1/2} + \sqrt{2\expe{\1_{\Dcal_{\I,2:d}^{d}}\abs{\barX_d-\barY_d}}(2 \pi \tilde{\sigma}^2_d)^{-1/2}} }  \right\} \eqsp.
\end{align}
By \Cref{lem:interm_prop_approx_ratio:G}, $ \PE[\1_{\Dcal_{\I,2:d}^{d}}|\barX_d - \barY_d|]$ goes to $0$ as $d$ goes to infinity, and by
\Cref{assum:diff:quadratic:G}\ref{assum:X6:G} $\lim_{d \to \plusinfty} \tilde{\sigma}^2_d =
\ell^2\expe{\Vdot(X)^2} $. Combining these results with
\eqref{eq:approx_first_term_decomp_martingale_3:G}, it follows that $\Erm^{\dim}$ goes to $0$
when $d$ goes to infinity.
\end{proof}

For all $n\ge 0$, define $\mcf_{n}^{\dim} = \sigma( \lbrace X_k^d,k\leq n \rbrace) $
and for all $\phi \in C_c^{\infty}(\rset,\rset)$,
\begin{multline}
\label{eq:def:martingale:G}
M^{\dim}_n(\phi) = \frac{\ell}{\sqrt{\dim}} \sum_{k=0}^{n-1} \phi'(X_{k,1}^d)\left\{Z_{k+1,1}^d \1_{\setAccept^d_{k+1}} - \CPE{Z_{k+1,1}^d \1_{\setAccept^d_{k+1}}}{\mcf_k^d} \right\} \\
+  \frac{\ell^2}{2 \dim} \sum_{k=0}^{n-1} \phi''(X_{k,1}^d)\left\{(Z_{k+1,1}^d)^2
\1_{\setAccept^d_{k+1}} - \CPE{(Z_{k+1,1}^d)^2 \1_{\setAccept^d_{k+1}}}{\mcf_k^d} \right\}  \eqsp.
\end{multline}

\begin{proposition}
\label{propo:decomposition_martingale:G}
Assume \Cref{assum:diff:quadratic:G} and \Cref{assum:Vdot:G} hold. Then, for all $s \leq t$ and all $\phi \in C_c^{\infty}(\rset,\rset)$,
\[
\lim_{\dim \to \plusinfty} \PE \left[ \abs{\phi(Y_{t,1}^{\dim} ) -\phi(Y_{s,1}^{\dim} ) - \int_s ^t \Lrm \phi(Y_{r,1}^{\dim}) \rmd r -\left(M^{\dim}_{\ceil{\dim t}}(\phi) - M^{\dim}_{\ceil{\dim s}}(\phi)\right)} \right] = 0\eqsp.
\]
\end{proposition}
\begin{proof}

Using the same decomposition as in the proof of \Cref{propo:decomposition_martingale}, we only need to prove that for all $1\le i\le 5$, $\lim_{d \to \plusinfty} \PE[ | T^d_i|] =0$, where
\begin{align*}
T^{\dim}_1 &=  \int_s ^t  \phi'(X^{\dim}_{\floor{\dim r },1}) \left( \ell \sqrt{\dim}  \
\CPE{Z_{\ceil{\dim r },1}^d \1_{\setAccept^d_{\ceil{\dim r }}}}{ \mcf_{\floor{\dim r}}^\dim}  + \frac{h(\ell)}{2}\Vdot(X^{\dim}_{\floor{\dim r },1})\right) \rmd r\eqsp,\\
T^{\dim}_2 & = \int_s ^t  \phi''(X^{\dim}_{\floor{\dim r },1}) \left( \frac{\ell^2}{2} \  \PE \left[ (Z_{\ceil{\dim r },1}^d)^2 \1_{\setAccept^d_{\ceil{\dim r }}}  \sachant{ \mcf_{\floor{\dim r}} ^d} \right] -  \frac{h(\ell)}{2} \right) \rmd r \eqsp,\\
T^{\dim}_3 & = \int_s ^t \left(\Lrm \phi(Y^{\dim}_{\floor{\dim r}/\dim,1}) - \Lrm \phi(Y^{\dim}_{r,1})\right) \rmd r\eqsp, \\
T^{\dim}_4& = \frac{\ell(\ceil{\dim t} - \dim t)}{\sqrt{\dim}} \phi'(X^{\dim}_{\floor{\dim t },1}) \left(  Z_{\ceil{\dim t },1}^d \1_{\setAccept^d_{\ceil{\dim t }}} -  \PE \left[ Z_{\ceil{\dim t },1}^d \1_{\setAccept^d_{\ceil{\dim t }}}  \sachant{ \mcf_{\floor{\dim t}} ^d} \right] \right) \\
& \; + \frac{\ell^2 (\ceil{\dim t} - \dim t)}{2 \dim } \phi''(X^{\dim}_{\floor{\dim t },1}) \left(  (Z_{\ceil{\dim t },1}^d)^2 \1_{\setAccept^d_{\ceil{\dim t }}} -  \PE \left[ (Z_{\ceil{\dim t },1}^d)^2 \1_{\setAccept^d_{\ceil{\dim t }}}  \sachant{ \mcf_{\floor{\dim t}} ^d} \right] \right)\eqsp, \\
T^{\dim}_5& = \frac{\ell(\ceil{\dim s} - \dim s)}{\sqrt{\dim}} \phi'(X^{\dim}_{\floor{\dim s },1}) \left(  Z_{\ceil{\dim s },1}^d \1_{\setAccept^d_{\ceil{\dim s }}} -  \PE \left[ Z_{\ceil{\dim s },1}^d \1_{\setAccept^d_{\ceil{\dim s }}}  \sachant{ \mcf_{\floor{\dim s}} ^d} \right] \right) \\
& \; + \frac{\ell^2 (\ceil{\dim s} - \dim s)}{2 \dim } \phi''(X^{\dim}_{\floor{\dim s },1}) \left(  (Z_{\ceil{\dim s },1}^d)^2 \1_{\setAccept^d_{\ceil{\dim s }}} -  \PE \left[ (Z_{\ceil{\dim s },1}^d)^2 \1_{\setAccept^d_{\ceil{\dim s }}}  \sachant{ \mcf_{\floor{\dim s}} ^d} \right] \right) \eqsp.
\end{align*}
First, as $\phi'$ and $\phi''$ are bounded, $\PE\left[\abs{T^{\dim}_4 }+\abs{ T^{\dim}_5} \right] \leq C \dim^{-1/2}$. Denote for all $r \in \ccint{s,t}$ and $d \geq 1$,
\begin{align*}
\Delta V_{r,i}^d &= V\left(X_{\floor{dr},i}^{\dim}\right) - V\left(X_{\floor{dr},i}^{\dim} +
  \ell\dim^{-1/2} Z^{\dim}_{\ceil{dr},i}\right) \\
 \Xi_r^d &= 1 \wedge
      \exp\left\{-\ell Z_{\ceil{\dim r },1}^{d}\Vdot(X_{ \floor{\dim r}
          ,1}^{d})/\sqrt{\dim} +\sum_{i=2}^d  \varb_{\I,i}^{d,\floor{\dim r}}\right\} \eqsp,
\end{align*}
where for all $k,i\ge 0$, $\varb_{\I,i}^{d,k} = \varb_{\I}^d(X_{k,i}^d,Z_{k+1,i}^d)$, and for all $x,z \in
\rset$, $b_{\I}^d(x,y)$ is given by \eqref{eq:bdi:G}. For all $k\ge 0$, $1\le i,j\le d$, define
\begin{align*}
\Dcal_{\I,j}^{d,k} &= \left\{X_{k,j}^d + \constSet \ell d^{-1/2}Z_{k+1,j}^d\in\I\right\} \cap \left\{X_{k,j}^d + ( 1-\constSet) \ell d^{-1/2}Z_{k+1,j}^d\in\I\right\}\\
\Dcal_{\I,i:j}^{d,k} &= \bigcap_{\ell=i}^j \Dcal_{\I,\ell}^{d,k}\eqsp.
\end{align*}
By the triangle inequality,
\begin{equation}
\label{eq:treatment_T1:G}
\abs{T_1}  \leq \int_s ^t  \left|\phi'(X^{\dim}_{\floor{\dim r },1})\right|(A_{1,r} + A_{2,r} + A_{3,r} + A_{4,r}) \rmd r \eqsp,
\end{equation}
where
\begin{align*}
\Pi_r^\dim &=  1 \wedge \exp
      \defEns{- \ell \dim^{-1/2} Z_{\ceil{dr},1}^{d}\Vdot(X_{\floor{dr}, 1}^{d}) +
        \sum_{i=2}^d \Delta V_{r,i}^d}\eqsp,\\
  A_{1,r} &= \left|\ell \sqrt{\dim} \PE \left[
    Z_{\ceil{\dim r },1}^d \left(\1_{\setAccept^d_{\ceil{\dim r }}} - 
    \1_{\Dcal_{\I,1:d}^{d,\floor{dr}} }\Pi_r^{\dim}\right) \sachant{ \mcf_{\floor{\dim r}} ^d} \right]\right|\eqsp,\\
  A_{2,r} &= \abs{\ell \sqrt{\dim} \  \PE \left[
    Z_{\ceil{\dim r },1}^d \1_{\Dcal_{\I,1:d}^{d,\floor{dr}} }\left( \Pi_r^{\dim}
      -\Xi_r^d
    \right)  \sachant{ \mcf_{\floor{\dim r}} ^d} \right]}\eqsp, \\
  A_{3,r} &= \abs{\ell \sqrt{\dim} \  \PE \left[
    Z_{\ceil{\dim r },1}^d \1_{\left(\Dcal_{\I,1:d}^{d,\floor{dr}}\right)^c}\Xi_r^d\sachant{ \mcf_{\floor{\dim r}} ^d} \right]}\eqsp, \\
A_{4,r}& =  \abs{\ell \sqrt{\dim}  \ \PE \left[ Z_{\ceil{\dim r },1}^d \Xi_r^{\dim} \sachant{ \mcf_{\floor{\dim r}} ^d} \right] + \Vdot(X^{\dim}_{\floor{\dim r },1}) h(\ell)/2 } \eqsp.
  \end{align*}
Since $t \mapsto 1 \wedge \exp(t)$ is $1$-Lipschitz,
\begin{align*}
\PE\left[\abs{A_{1,r}^{\dim}}\right] &\le \ell\sqrt{d} \PE\left[\1_{\Dcal_{\I,1:d}^{d,\floor{dr}}}\left|Z_{\ceil{\dim r },1}^d\right| \left|\Delta V_{r,1}^d - \ell \dim^{-1/2} Z_{\ceil{dr},1}^{d}\Vdot(X_{\floor{dr}, 1}^{d})\right|\right]\eqsp,\\
&\le \ell\sqrt{d} \PE\left[\1_{\Dcal_{\I,1}^{d,\floor{dr}}}\left|Z_{\ceil{\dim r },1}^d\right| \left|\Delta V_{r,1}^d - \ell \dim^{-1/2} Z_{\ceil{dr},1}^{d}\Vdot(X_{\floor{dr}, 1}^{d})\right|\right]\eqsp,\\
&\le \ell\sqrt{d} \PE\left[\left|Z_{\ceil{\dim r },1}^d\right| \left|\1_{\Dcal_{\I,1}^{d,\floor{dr}}}\Delta V_{r,1}^d - \ell \dim^{-1/2} Z_{\ceil{dr},1}^{d}\Vdot(X_{\floor{dr}, 1}^{d})\right|\right]
\end{align*}
and $\PE[\abs{A_{1,r}^{\dim}}]$ goes to
$0$ as $d\to \plusinfty$
for almost all $r$ by \Cref{lem:integrated-DQM:G}\ref{lem:integrated-DQM-Lp:G}. So by
the Fubini theorem, the
first term  in \eqref{eq:treatment_T1:G} goes to $0$ as $d \to \plusinfty$.
For $A_{2,r}^{\dim}$, note that
\[
A_{2,r} \le \abs{\ell \sqrt{\dim} \  \PE \left[
    Z_{\ceil{\dim r },1}^d \1_{\Dcal_{\I,2:d}^{d,\floor{dr}} }\left(\Pi_r^{\dim}
      -\Xi_r^d
    \right)  \sachant{ \mcf_{\floor{\dim r}} ^d} \right]}\eqsp.
\]
Then, by \cite[Lemma 6]{jourdain:lelievre:miasojedow:2015},
\begin{multline*}
\expe{\abs{A_{2,r}^{\dim}}} \leq \PE\left[\left|\ell^2\Vdot(X_{\floor{dr},1}^d)\1_{\Dcal_{\I,2:d}^{d,\floor{dr}} }\left\{\mathcal{G}\parenthese{\frac{\ell^2\Vdot(X_{\floor{dr},1}^d)^2}{d}, 2\sum_{i=2}^d
        \Delta V_{r,i}^d }\right.\right.\right.\\
        \left.\left.\left.-\mathcal{G}\parenthese{\frac{\ell^2\Vdot(X_{\floor{dr},1}^d)^2}{d}, 2 \sum_{i=2}^d  \varb_{\I, i}^{d, \floor{\dim r}}}\right\}\right|\right]  \eqsp,
\end{multline*}
where $\mathcal{G}$ is defined in \eqref{eq:defG}. By  \Cref{lem:approx_first_term_decomposition_martingale:G}, this expectation  goes to zero
when $d$ goes to infinity. Then by the Fubini theorem and the Lebesgue dominated
convergence theorem, the second term of \eqref{eq:treatment_T1:G} goes $0$ as $d \to
\plusinfty$. On the other hand, by \Cref{assum:diff:quadratic:G}\ref{assum:proba:G} and Holder's inequality applied with $\alpha= 1/(1-2/\gamma)>1$, for all $1\le j \le 4$,
\begin{align*}
\expe{\abs{A_{3,r}^{\dim}}}& \leq \ell \sqrt{d}\left(\PE\left[\left|Z_{\ceil{\dim r },1}^d\right|\1_{\left(\Dcal_{\I,1}^{d,\floor{dr}}\right)^c}\right]+\sum_{i=2}^d \PE\left[\1_{\left(\Dcal_{\I,i}^{d,\floor{dr}}\right)^c}\right]\right)\eqsp,\\
&\le \ell\sqrt{d}\left(\PE\left[|Z_{m_j,1}^{\dim}|^{\alpha/(\alpha-1)}\right]^{(\alpha-1)/\alpha}d^{-\gamma/(2\alpha)} + d^{1-\gamma/2}\right)\le Cd^{3/2-\gamma/2}
\end{align*}
and $\PE[\abs{A_{3,r}^{\dim}}]$ goes to
$0$ as $d\to \plusinfty$
for almost all $r$. Define 
\[
\bar V_{\dim,1} = \sum_{i=1}^{d} \Vdot(X_{\floor{\dim r},i}^{d})^2\quad\mbox{and}\quad\bar V_{\dim,2} = \bar V_{\dim,1} -  \Vdot(X_{\floor{\dim r},1}^{d})^2\eqsp.
\]
For the last term, by \cite[Lemma 6]{jourdain:lelievre:miasojedow:2015}:
\begin{multline}
 \ell \sqrt{d} \ \PE \left[ Z_{\ceil{\dim r },1}^d \Xi_r^{\dim} \sachant{
    \mcf_{\floor{\dim r}} ^d} \right]  = -\ell^2\Vdot(X_{ \floor{\dim r}  ,1}^{d})  \\
\label{eq:treatment_T1_1:G}
 \times\mathcal{G}\left(
  \frac{\ell^2}{\dim}\bar V_{\dim,1}, \left\lbrace \frac{\ell^2}{2\dim}\bar V_{\dim,2} -4(d-1) \PE\left[\1_{\Dcal_{\I}}\zeta^{\dim}(X,Z)\right] \right\rbrace
  \right) \eqsp,
\end{multline}
where  $\Dcal_{\I}= \left\{X+\ell d^{-1/2}Z\in \I\right\}$, $X$ is distributed according to $\pdf$ and $Z$ is a standard Gaussian random variable independent of $X$. As $\G$ is continuous on $\rset_+ \times \rset \setminus \defEns{0,0}$ (see \cite[Lemma 2]{jourdain:lelievre:miasojedow:2015}), by \Cref{assum:diff:quadratic:G}\ref{assum:X6:G}, \Cref{lem:mean:zeta:G} and the law of large numbers, almost surely,
\begin{multline}
\label{eq:limG:G}
\lim_{d \to \plusinfty} \ell^2\mathcal{G}\left(\ell^2\bar V_{\dim,1} /d, \ell^2\bar V_{\dim,2}/(2d) -4(d-1) \PE\left[\1_{\Dcal_{\I}}\zeta^{\dim}(X,Z)\right] \right) \\ = \ell^2\mathcal{G}\left(\ell^2\PE_{}[\Vdot(X)^2],\ell^2\PE_{}[\Vdot(X)^2]\right) = h(\ell)/2\eqsp,
\end{multline}
where $h(\ell)$ is defined in \eqref{eq:defhK}. Therefore by Fubini's Theorem,
\eqref{eq:treatment_T1_1:G} and Lebesgue's dominated convergence theorem, the last term of
\eqref{eq:treatment_T1:G} goes to $0$ as $d$ goes to infinity. The proof for $T_2^d$ follows the same lines. By the triangle inequality,
\begin{multline}
\label{eq:treatment_T2:G}
\abs{T_2^{d}}  \leq \abs{\int_s ^t  \phi''(X^{\dim}_{\floor{\dim r },1})(\ell^2/2) \  \PE \left[ (Z_{\ceil{\dim r },1}^d )^2 \left( \1_{\setAccept^d_{\ceil{\dim r }}} -\Xi_r^{\dim} \right) \sachant{ \mcf_{\floor{\dim r}} ^d} \right]  \rmd r}\\
+ \abs{\int_s ^t  \phi''(X^{\dim}_{\floor{\dim r },1}) \left( (\ell^2/2) \ \PE \left[
      (Z_{\ceil{\dim r },1}^d)^2 \Xi_r^{\dim} \sachant{ \mcf_{\floor{\dim r}} ^d} \right]
    -  h(\ell)/2 \right) \rmd r} \eqsp.
\end{multline}
 By Fubini's Theorem,  Lebesgue's dominated convergence theorem and \Cref{lem:approx_ratio:G}, the expectation of the first term goes to zero when $d$ goes to infinity. For the second term, by \cite[Lemma 6 (A.5)]{jourdain:lelievre:miasojedow:2015},
\begin{multline}
\label{eq:treatment_T2_1:G}
( \ell^2/2) \PE \left[ (Z_{\ceil{\dim r },1}^d)^2 1 \wedge \exp\left\{ -\frac{\ell Z_{\ceil{\dim r },1}^{d}}{\sqrt{\dim}}\Vdot(X_{ \floor{\dim r}  ,1}^{d}) +\sum_{i=2}^d  \varb_{\I, i}^{d,\floor{\dim r}} \right\} \sachant{ \mcf_{\floor{\dim r}} ^d} \right] \\
= (B_1 + B_2 - B_3)/2 \eqsp,
\end{multline}
where
\begin{align*}
B_1 & = \ell^2\Gamma \left(  \ell^2\bar V_{\dim,1} /d, \ell^2\bar V_{\dim,2}/(2d)-4(d-1) \PE_{}\left[\1_{\Dcal_{\I}}\zeta^{\dim}(X,Z)\right]\right)\eqsp,\\
B_2 & = \frac{\ell^4 \Vdot(X_{ \floor{\dim r}  ,1}^{d})^2}{d} \mathcal{G}\left(
  \ell^2\bar V_{\dim,1} /d, \ell^2\bar V_{\dim,2}/(2d) -4(d-1) \PE_{}\left[\1_{\Dcal_{\I}}\zeta^{\dim}(X,Z)\right]
  \right)\eqsp,\\
B_3 & = \frac{\ell^4 \Vdot(X_{ \floor{\dim r}  ,1}^{d})^2}{d} \left(2 \pi \ell^2\bar V_{\dim,1} /d\right)^{-1/2}\\
&\hspace{3cm}\times \exp\left\{- \frac{\left[-2(d-1)\PE_{}[\1_{\Dcal_{\I}}\zeta^{\dim}(X,Z)] + (\ell^2/(4d)) \bar V_{\dim,2}\right]^2}{2\ell^2\bar V_{\dim,1} /d }\right\}\eqsp,\\
\end{align*}
where $\Gamma$ is defined in \eqref{eq:defGamma}. As $\Gamma$ is continuous on $\rset_+ \times \rset \setminus \defEns{0,0}$ (see \cite[Lemma 2]{jourdain:lelievre:miasojedow:2015}), by \Cref{assum:diff:quadratic:G}\ref{assum:X6:G}, \Cref{lem:mean:zeta:G} and the law of large numbers, almost surely,
\begin{multline}
\label{eq:first_treat_gamma_T2:G}
\lim_{d \to \plusinfty}
\ell^2 \Gamma\left(
  \ell^2\bar V_{\dim,1} /d, \left\lbrace  \ell^2 \bar V_{\dim,2}/(2d) -4(d-1) \PE_{}\left[\1_{\Dcal_{\I}}\zeta^{\dim}(X,Z)\right] \right\rbrace
  \right) \\= \ell^2\Gamma\left(\ell^2\PE_{}[\Vdot(X)^2],\ell^2\PE_{}[\Vdot(X)^2]\right) = h(\ell) \eqsp.
  \end{multline}
By \Cref{lem:mean:zeta:G}, by \Cref{assum:diff:quadratic:G}\ref{assum:X6:G} and the law of large numbers, almost surely,
  \[
  \lim_{d\to\plusinfty}\exp\left\{- \frac{\left[-2(d-1)\PE_{}[\1_{\Dcal_{\I}}\zeta^{\dim}(X,Z)] + (\ell^2/(4d)) \bar V_{\dim,2}\right]^2}{2\ell^2\bar V_{\dim,1} /d }\right\}= \exp\left\{-\frac{\ell^2}{8}\PE_{}[\Vdot(X)^2]\right\}\eqsp.
  \]
  Then, as $\mathcal{G}$ is bounded on $\rset_+ \times \rset$,
  \begin{equation}
  \label{eq:second_treat_gamma_T2:G}
  \lim_{d\to\plusinfty} \PE\left[\left|\int_s ^t  \phi''(X^{\dim}_{\floor{\dim r },1}) \left( B_2-B_3 \right) \rmd r\right|\right] = 0\eqsp.
  \end{equation}
   Therefore, by Fubini's Theorem,
   \eqref{eq:treatment_T2_1:G}, \eqref{eq:first_treat_gamma_T2:G}, \eqref{eq:second_treat_gamma_T2:G}
   and Lebesgue's dominated convergence theorem, the second term of
   \eqref{eq:treatment_T2:G} goes to $0$ as $d$ goes to infinity. The proof for $T_3^d$ follows exactly the same lines as the proof of \Cref{propo:decomposition_martingale}.
\end{proof}

\begin{proof}[Proof of \Cref{theo:diffusion_limit_RMW:G}]
Using \Cref{prop:tight:G}, \Cref{propo:reduction_martingale_problem:G} and
\Cref{propo:decomposition_martingale:G}, the proof follows the same lines as the proof of \Cref{theo:diffusion_limit_RMW}.
\end{proof}

\subsection*{Acknowledgment}
The work of A.D. and E.M. is supported by the Agence Nationale de la Recherche, under grant ANR-14-CE23-0012 (COSMOS).

\bibliographystyle{plain}
\bibliography{./scaling_HM_MALA_bib}
\end{document}